\documentclass[10pt,reqno,twoside]{article}
\usepackage{mysty}

\usepackage{wrapfig}
\usepackage{mdwlist}

\usepackage{fullpage}
\usepackage{patchcmd,pifont}


\usepackage{enumitem}

\SetLabelAlign{parright}{\parbox[t]{\labelwidth}{\raggedleft{#1}}}
\setlist[description]{style=multiline,topsep=4pt,align=parright}

\makeatletter
\let\reftagform@=\tagform@
\def\tagform@#1{\maketag@@@{(\ignorespaces\textcolor{black}{#1}\unskip\@@italiccorr)}}
\newcommand{\iref}[1]{\textup{\reftagform@{\tcr{\ref{#1}}}}}
\makeatother

\begin{document}
\title{An SDE Perspective on Stochastic Inertial Gradient Dynamics with Time-Dependent Viscosity and Hessian-Driven Damping}
\author{Rodrigo Maulen-Soto \thanks{Normandie Universit\'e, ENSICAEN, UNICAEN, CNRS, GREYC, France. E-mail: rodrigo.maulen@ensicaen.fr} \and
Jalal Fadili\thanks{Normandie Universit\'e, ENSICAEN, UNICAEN, CNRS, GREYC, France. E-mail: Jalal.Fadili@ensicaen.fr} \and Hedy Attouch\thanks{IMAG, CNRS, Universit\'e Montpellier, France. E-mail: hedy.attouch@umontpellier.fr} \and Peter Ochs\thanks{Department of Mathematics and Computer Science, Saarland University, Germany, E-mail: ochs@math.uni-saarland.de}
}
\date{}
\maketitle

\begin{abstract}
Our approach is part of the close link between continuous dissipative dynamical systems and optimization algorithms. Motivated by solving stochastic convex optimization problems in real Hilbert spaces, we propose a class of stochastic inertial differential equations which are driven by the gradient of the objective function. This will provide a general mathematical framework for analyzing fast optimization algorithms with stochastic gradient input. Our goal is to develop convergence guarantees for second-order stochastic differential equations in time, incorporating a viscous time-dependent damping and a Hessian-driven damping. To develop this program, we rely on stochastic Lyapunov analysis. Assuming a square-integrability condition on the diffusion term times a function dependent on the viscous damping, and that the Hessian-driven damping is a positive constant, our first main result shows that almost surely, there is convergence of the objective values with a fast convergence rate in expectation. Besides, in the case where the Hessian-driven damping is zero, we get fast convergence of the values in expectation and in almost sure sense, and we also prove almost sure weak convergence of the trajectory. We provide a comprehensive complexity analysis by establishing several new convergence rates in expectation and in almost sure sense for the convex and strongly convex case. 
\end{abstract}




\section{Introduction}\label{sec:intro}

\subsection{Problem Statement}\label{problem_statement}
In this paper, we consider the minimization problem
\begin{equation}\label{P}\tag{P}
    \min_{x\in \H} f(x),
\end{equation}

\noindent where $\H$ is a real Hilbert space and the objective function $f: \H \to \R$ satisfies the following standing assumptions:
\begin{align}\label{H0}
\begin{cases}
\text{$f$ is convex and continuously twice differentiable with $L$-Lipschitz continuous gradient}; \\
\calS \eqdef \argmin (f)\neq\emptyset. \tag{$\mathrm{H}_0$} 
\end{cases}
\end{align} 

\paragraph{First-order in time systems.}
To solve \eqref{P}, a fundamental dynamic is the gradient flow system:
\begin{equation}
\begin{cases}\label{gf}\tag{GF}
\begin{aligned}
\dot{x}(t)+\nabla f(x(t))&= 0,\quad t>t_0;\\
x(t_0)&=x_0.
\end{aligned}
\end{cases}
\end{equation}
The gradient system \eqref{gf} is a dissipative dynamical system, whose study dates back to Cauchy \cite{Cauchy} in finite dimensions. It is key in optimization, as it converts minimizing $f$ into analyzing the asymptotic behavior of the trajectories of \eqref{gf}. In fact, Brezis, Baillon, and Bruck showed in the 1970s that if $\argmin(f)$ in \eqref{P} is non-empty, then every solution trajectory of \eqref{gf} converges weakly to a point in $\argmin(f)$, with a convergence rate of $\mathcal{O}(t^{-1})$ (even $o(t^{-1})$) on the function values.

The Euler forward (or Euler-Maruyama) discretization of \eqref{gf}, with stepsizes $h_k>0$, yields the gradient descent scheme \begin{equation}\label{GD}\tag{GD} x_{k+1}=x_k-h_k \nabla f(x_k). \end{equation} Under assumption \eqref{H0} and for $(h_k){k \in \N}\subset ]0,2/L[$, we have $f(x_k)-\min f=\calO(1/k)$ (even $o(1/k)$) and the iterates $(x_k)_{k\in\N}$ converge weakly to a point in $\argmin(f)$. This rate can be improved under additional geometric conditions on $f$, such as error bounds or the Kurdyka-{\L}ojasiewicz property in the convex case \cite{BolteKLComplexity16}.

\paragraph{Second-order in time systems: Key role of inertia.}
Second-order inertial dynamical systems have been introduced to accelerate convergence compared to \eqref{gf}. An abundant literature has been devoted to the study of inertial dynamics \begin{equation}\tag{$\mathrm{IGS}_{\gamma}$}\label{AVD} \ddot{x}(t)+\gamma(t)\dot{x}(t)+\nabla f(x(t))=0,\quad t>t_0. \end{equation} Several authors (e.g., \cite{cabot, engler}) highlighted that an asymptotically vanishing viscosity coefficient $\gamma(t)$ is crucial for acceleration. Most of the literature focuses on the case $\gamma(t)=\frac{\alpha}{t}$, stemming from the seminal work \cite{su} which established an $\mathcal{O}(1/t^2)$ convergence rate on the values for $\alpha=3$, linking it to Nesterov's accelerated gradient method \cite{1983}. Further research shows that $\alpha\geq3$ is required for the $\mathcal{O}(1/t^2)$ rate \cite{11}, while $\alpha>3$ yields an improved $o(1/t^2)$ rate and weak convergence of the trajectory \cite{faster1k2, may}. Another remarkable case of \eqref{AVD} is the Heavy Ball with Friction (HBF) method, where $\gamma(t)$ is constant, as introduced by Polyak \cite{speedingup}. In the strongly convex setting, the trajectory converges exponentially with an optimal rate for a well-chosen constant $\gamma$ \cite{heavyb}.

\paragraph{Geometric Hessian-driven damping.}
Because of the inertial aspects, and the asymptotic vanishing viscous damping coefficient, \eqref{AVD} may exhibit many small oscillations which are not desirable from an optimization point of view. To remedy this, a powerful tool consists in introducing a geometric damping driven by the Hessian of $f$ into the dynamic. This yields the Inertial System with Explicit Hessian-driven Damping, first proposed in \cite{AABR} (\tcb{and further studied in \cite{APR2,ACFR,ABCR}}):
\begin{equation} \label{ISEHD}\tag{ISEHD}
\qquad \ddot{x}(t) + \gamma(t)\dot{x}(t) +  \beta (t) \nabla^2  f (x(t)) \dot{x} (t) + \nabla  f (x(t)) = 0,
\end{equation}
and the Inertial System with Implicit Hessian-driven Damping \cite{alecsa,MJ}:
\begin{equation}
\label{ISIHD}\tag{ISIHD}
\ddot{x}(t)+\gamma(t)\dot{x}(t)+\nabla f(x(t)+\beta(t)\dot{x}(t)) = 0 .
\end{equation}
The rationale behind the use of the term ``implicit'' in \eqref{ISIHD} comes from a Taylor expansion of the gradient term (as $t \to +\infty$ we expect $\dot{x}(t) \to 0$). Following the physical interpretation of these ODEs, we call the non-negative parameters $\gamma$ and $\beta$ as the viscous and geometric damping parameters, respectively. 

At first glance, the presence of the Hessian in \eqref{ISEHD} may seem to entail numerical difficulties. However, this is not the case as the term $\nabla^2  f (x(t)) \dot{x} (t)$ is nothing but the time derivative of $t \mapsto \nabla  f (x(t))$. This explains why the time discretization of this dynamic provides efficient first-order algorithms \cite{ACFR}. However, in our stochastic setting---the focus of this paper--- the proposed approach applies only to the implicit form of the Hessian-driven damping. In this context, we do not have direct access to $\nabla f$, and instead model errors with a continuous It\^o martingale $M(t)$. The explicit form would involve the derivative of $\nabla f(X(t)) + M(t)$, which is meaningless since non-constant martingales are almost surely non-differentiable. Hence, we will only consider \eqref{ISIHD}.

\subsection{Motivations}
In many practical situations, the gradient evaluation is subject to stochastic errors---either because each iteration is costly, requiring cheap random approximations, or due to other exogenous factors. The continuous-time approach through stochastic differential equations (SDE) is a powerful way to model these errors in a unified way, and stochastic algorithms can then be viewed as time- discretizations. In fact, several recent works have used the SDE perspective to model SGD-type algorithms motivated by various reasons; (see \eg \cite{optch,continuous,dif,schrodinger,valid,cycle,mertikopoulos_staudigl_2018,mio,sdemodel,mio2}). The continuous-time perspective offers a deep insight and unveils the key properties of the dynamic without being tied to a specific discretization.

In this context, a natural SDE is
\begin{equation}
    \label{CSGDintro}\tag{SGF}
    \begin{cases}
     dX(t)&=-\nabla f(X(t))+\sigma(t,X(t))dW(t), \quad t>t_0;\\
    X(t_0)&=X_0,
    \end{cases}
\end{equation}
defined over a complete filtered probability space $(\Omega,\calF, (\calF_t)_{t \geq t_0}, \mathbb{P})$, where the diffusion (volatility) term $\sigma: \R^+ \times \H \to \mathcal{L}_2(\K;\H)$ is a measurable function, $\H,\K$ are real separable Hilbert spaces, $W$ is a $\calF_t$-adapted $\K$-valued cylindrical Brownian motion, and the initial data $X_0$ is an $\calF_{t_0}-$measurable $\H$-valued random variable. We will elaborate more on these notations and concepts in Section~\ref{sec:notation}. \eqref{CSGDintro} is a stochastic counterpart of \eqref{gf} where the error is modeled as a stochastic integral with respect to the measure defined by a continuous It\^{o} martingale. 

\paragraph{SDE modeling of SGD.}\label{sec:sdesgd}
To simplify the discussion in this section, we restrict our attention to the finite-dimensional case ($\H=\R^d$, $\K=\R^m$). In various applications, specially in machine learning, the Stochastic Gradient Descent (SGD) is a powerful alternative to gradient descent, and consists in replacing the full gradient computation by a cheaper random version, serving as an unbiased estimator. The SGD updates the iterates according to
\begin{equation}\tag{SGD}\label{sgdd}
x_{k+1}=x_k-h(\nabla f(x_k) + e_k),
\end{equation}
where $h>0$ is the stepsize and $e_k$ is the random noise term on the gradient at the $k$-th iteration. As such, \eqref{sgdd} can be viewed as instance of the Robbins-Monro stochastic approximation algorithm \cite{rob}. When the objective takes the form $f(x) = \EE[\hat{f}(x,\xi)]$, where the expectation is w.r.t. to the random variable $\xi$ the single-batch version of SGD reads
\begin{equation}\label{sgdmini}\tag{$\mathrm{SGD_{SB}}$}
x_{k+1}=x_k-h\nabla \hat{f}(x_k,\xi_k),
\end{equation}
where $\seq{\xi_k}$ are i.i.d. random variables with the same distribution as $\xi$. Of course, \eqref{sgdmini} is an instance of \eqref{sgdd} with $e_k = \nabla \hat{f}(x_k,\xi_k) - \nabla f(x_k)$.

A natural continuous-time model for these stochastic algorithms is \eqref{CSGDintro}. In particular, when the noise $e_k$ in \eqref{sgdd} is Gaussian (i.e., $e_k\sim\mathcal{N}(0,\sigma_k I_d)$), it has been shown that \eqref{CSGDintro} accurately approximates \eqref{sgdd} \cite[Proposition~2.1]{sdemodel}. Similar approximations hold for \eqref{sgdmini} under appropriate conditions, see \eg \cite{Mandt16,optch,continuous,dif,schrodinger,valid,cycle,Xie21,latz,hessianaware}.

Overall, the continuous-time SDE framework provides a powerful tool for analyzing SGD-type algorithms. By modeling the dynamics with \eqref{CSGDintro}, one can exploit the rich theoretical foundations of SDEs, Itô calculus, and measure theory to uncover the essential properties of an algorithm. This perspective not only captures the underlying behavior of the dynamics in a way that is independent of a particular discretization, but it also enables the transfer of convergence results from the continuous model directly to the discrete setting. In essence, many stability and convergence properties that are proven for \eqref{CSGDintro} can be used to predict and explain the behavior of the corresponding discrete algorithms, including their ability to escape saddle points in non-convex scenarios. For a more detailed discussion of these aspects, we refer the reader to \cite{mio3}.

\paragraph{Second-order SDE modeling of inertial SGD.}
Using a lifting argument to get an equivalent first-order reformulation, a natural reformulation of \eqref{ISIHD} yields the differential equation
\begin{equation}
\begin{cases}\label{ISIHDN}\tag{$\mathrm{ISIHD_{R}}$}
\begin{aligned}
\dot{x}(t)&=v(t), \quad t>t_0; \\
\dot{v}(t)&= -[\gamma(t)v(t)+\nabla  f(x(t)+\beta(t)v(t))], \quad t>t_0;\\
x(t_0)&=x_0,\quad \dot{x}(t_0)=v_0.
\end{aligned}
\end{cases}
\end{equation}
In this setting, we can model the associated errors using a stochastic integral with respect to the measure defined by a continuous It\^o martingale. This entails the following stochastic differential equation (SDE for short), which is the stochastic counterpart of \eqref{ISIHDN}:
\begin{equation}\label{ISIHD-S}\tag{$\mathrm{S-ISIHD}$}
\begin{cases}
d{X(t)}&={V(t)}dt,  \nonumber\\
d{V(t)}&= -\gamma(t){V(t)}dt-\nabla f({X(t)}+\beta(t){V(t)})dt+\sigma(t,{X(t)}+\beta(t){V(t)})dW(t),\\
{X(t_0)}&={X_0}, \quad {V(t_0)}={V_0}. \nonumber
\end{cases}
\end{equation}


Extending what we have discussed above for \eqref{CSGDintro} as a good model of \eqref{sgdd}, it can be shown that \eqref{ISIHD-S} is a provably good model of the stochastic inertial algorithm 
\begin{equation}\label{stochnest}
    \begin{cases}
        X_{k+1}&=X_k+hV_k,\\
        V_{k+1}&=(1-\gamma_k h)V_k-h\nabla f(X_k+\beta_k V_k) + \sqrt{h} \sigma_k G_k,
    \end{cases}
\end{equation}
where $G_k \sim \mathcal{N}(0,I_d)$, $\gamma_k = \gamma(t_k)$, $\sigma_k = \sigma(t_k,X_k+\beta_k V_k)$ and $\beta_k = \beta(t_k)$. More precisely, this algorithm is obtained by simple Euler-Maruyama discretization of \eqref{ISIHD-S}. It has been shown in \cite[Appendix~A]{mio3} that \eqref{stochnest} is a consistent discretization of \eqref{ISIHD-S} with an approximation error that vanishes at the rate $\mathcal{O}(h)$. This is much better than the approximation rate for \eqref{ISIHD} which is only $\mathcal{O}(\sqrt{h})$. As a consequence, this justifies that the continuous-time dynamics \eqref{ISIHD-S} is a better proxy of the stochastic inertial algorithm \eqref{stochnest} and opens the door to new insights in the behavior of such algorithms. In turn, the convergence properties of such algorithms can be easily derived from those of \eqref{ISIHD-S} with minimal effort. For instance, when $f$ is smooth with Lipschitz-continuous gradient, it is easy to see from the descent lemma that
\[
\EE[f(X_k) - \min f] = \EE[f(X(kh)) - \min f] + \calO(h),
\]
where $X_k$ are the iterates of \eqref{stochnest} and $X(t)$ the solution to \eqref{ISIHD-S}. This means that any rate proved on $\EE[f(X(t))-\min f]$ can be directly transferred to $\EE[f(X_k)-\min f]$.

\subsection{Objectives and Contributions}
In this work,  our goal is to provide a general mathematical framework for analyzing fast gradient-based optimization algorithms with stochastic gradient input.
For this, we will study second-order stochastic differential equations in time featuring inertia that is crucial for acceleration, and whose drift term is the gradient of the function to be minimized. In this context, considering inertial dynamics combining a time-dependent viscosity coefficient and geometric damping is a key property to obtain fast convergent methods while reducing oscillations.

More precisely, we study the stochastic dynamics \eqref{ISIHD-S} and its long-time behavior in order to solve \eqref{P}. We conduct a new analysis using specific and careful arguments that are much more involved than in the deterministic case. To get some intuition, keeping the discussion informal at this stage, let us first identify the assumptions needed to expect that the position state of \eqref{ISIHD-S} ``approaches'' $\argmin (f)$ in the long run. In the case where $\H=\K$, $\gamma(\cdot)\equiv\gamma>0,\beta\equiv 0,$ and $\sigma=\tilde{\sigma} I_{\H}$, where $\tilde{\sigma}$ is a positive real constant. Under mild assumptions one can show that \eqref{ISIHD-S} has a unique invariant distribution $\pi_{\tilde{\sigma}}$ in $(x,v)$ with density proportional to $ \exp\pa{-\frac{2\gamma}{\tilde{\sigma}^2}\left(f(x)+\frac{\Vert v\Vert^2}{2}\right)}$, see e.g., \cite[Proposition~6.1]{pavliotis}. 
Clearly, as $\tilde{\sigma} \to 0^+$, $\pi_{\tilde{\sigma}}$ gets concentrated around $\argmin (f)\times \{0_{\H}\}$, with $\lim_{\tilde{\sigma} \to 0^+} \pi_{\tilde{\sigma}}(\argmin (f) \times  \{0_{\H}\}) = 1$; see also Section~\ref{subsec:relwork} for further discussion. Motivated by these observations and the fact that we aim to exactly solve \eqref{P}, our paper will then mainly focus on the case where $\sigma(\cdot,x)$ vanishes fast enough as $t \to +\infty$ uniformly in $x$.

Our main contributions are summarized as follows:
\begin{itemize}
\item We will develop a Lyapunov analysis to obtain convergence rates and integral estimates, in expectation and almost sure sense, in the general case of coefficients $\gamma(t)$ and $\beta(t)$.

\item We will study two instances where the rates can be made even more explicit highlighting acceleration phenomena: when $\gamma(t)=\frac{\alpha}{t},\beta(t)=\beta_0+\frac{\gamma_0}{t}$, and when $\gamma$ is decreasing and vanishing with vanishing derivative, and $\beta(t)$ is constant.

\item In the case where the coefficient $\beta(t)$ is zero, we show that under some hypotheses, we have almost sure weak convergence of the trajectory, convergence rates, and integral estimates. As a special case, we focus on viscous damping coefficient $\gamma(t)=\frac{\alpha}{t^r}, r\in [0,1], \alpha\geq 1-r$.

\end{itemize}

\subsection{Relation to prior work}\label{subsec:relwork}
\paragraph{Kinetic diffusion dynamics for sampling} Let us consider \eqref{ISIHD-S} in the case where $\H=\K=\R^d$, $\gamma(t)=\gamma>0, \beta(t)=0$ and $\sigma=\sqrt{2\gamma}I_d$. Then one recovers the kinetic Langevin diffusion (or second-order Langevin process). In this case, the continuous-time Markov process $(X(t),V(t))$ is positive recurrent and has a unique invariant distribution which has the density proportional to $ \exp\left(-f(x)-\frac{\|v\|^2}{2}\right)$ with respect to the Lebesgue measure on $\R^{2d}$. Time-discretized versions of this Langevin diffusion process have been studied in the literature to (approximately) sample from a distribution proportional to $ \exp(-f(x))$ with asymptotic and non-asymptotic convergence guarantees in various topologies and under various conditions have been studied; see \cite{underdamped,isthere,Dalalyan2019BoundingTE} and references therein.

\paragraph{Inexact inertial gradient systems} There is an abundant literature regarding the dynamics \eqref{ISIHD} and \eqref{ISEHD}, either in the exact case or with errors but only deterministic ones; see \cite{alecsa, hessianpert,27,35,8,11,20,34,37,19,9,10}). We are not aware of any such work in the stochastic case. On a technical level, our Lyapunov analysis will be inspired by that in \cite{hessianpert} and \cite{AC1}. Only a few papers have been devoted to studying the second-order in-time inertial stochastic gradient systems with viscous damping, \ie stochastic versions of \eqref{AVD}, either with vanishing damping $\gamma(t)=\alpha/t$ or constant damping $\gamma(t)$ (stochastic HBF); see \eg \cite{sdemodel, rolememory, gadat1}. For instance, \cite{sdemodel} provide asymptotic $\mathcal{O}(1/t^2)$ convergence rate on the objective values in expectation under integrability conditions on the diffusion term as well as other rates under additional geometrical properties of of the objective.
The corresponding stochastic algorithms for these two choices of $\gamma$, whose mathematical formulation and analysis is simpler, have been the subject of active research work; see \eg \cite{Lin2015,Frostig2015,Jain2018,AR,AZ,Yang2016,Gadat2018,Loizou2020,Laborde2020,LanBook,Sebbouh2021,Driggs22,Wu24,AFKstochastic24}.

\paragraph{Time scaling and averaging} A first-order SDE to solve \eqref{P} has been thoroughly studied in \cite{mio}; see also \cite{mio2} for the non-smooth setting. This SDE has the form
\begin{equation}\label{mmm}
\begin{cases}
\begin{aligned}
    dX(t)&= -\nabla f(X(t))dt+\sigma(t,X(t))dW(t), \quad t\geq t_0,\\
    X(t_0)&=X_0.
    \end{aligned}
    \end{cases}
\end{equation}
Our work here is a natural extension of \cite{mio,mio2} to second-order systems hence allowing for accelerated rates. By leveraging the time scaling and averaging trick pioneered in \cite{fast} to go from \eqref{gf} to \eqref{ISIHD}, we were able in \cite{mio3} to transfer the results in \cite{mio} for \eqref{mmm} to \eqref{ISIHD-S}. The approach in \cite{mio3} avoids, in particular, going through an intricate Lyapunov analysis for \eqref{ISIHD-S}. However, this can only be done for a specific choice of $\beta$ (related to the viscous damping function $\gamma$). On the other hand, getting full advantage of the Hessian-driven damping term requires a careful choice of $\beta$, that is possibly independent of $\gamma$. Thus handling flexible and general choices of $\beta$ $\gamma$ makes necessary to go through a dedicated Lyapunov analysis for \eqref{ISIHD-S}.

\section{Notation and Preliminaries}\label{sec:notation}

We will use the following shorthand notations:  Given $n\in\N$,  $[n]\eqdef \{1,\ldots,n\}$. {Consider $\H,\K$ real separable Hilbert spaces endowed with the inner product $\langle\cdot,\cdot\rangle_{\H}$ and $\langle\cdot,\cdot\rangle_{\K}$, respectively, and norm $\Vert \cdot\Vert_{\H}=\sqrt{\langle \cdot,\cdot \rangle_{\H}}$ and $\Vert \cdot\Vert_{\K}=\sqrt{\langle \cdot,\cdot \rangle_{\K}}$, respectively (we omit the subscripts $\H$ and $\K$ for the sake of clarity). $I_{\H}$ is the identity operator on $\H$. $\calL(\K;\H)$ is the space of bounded linear operators from $\K$ to $\H$, $\calL_1(\K)$ is the space of trace-class operators, and $\calL_2(\K;\H)$ is the space of bounded linear Hilbert-Schmidt operators from $\K$ to $\H$}. For $M\in\calL_1(\K)$, is trace is defined by
\[
\tr(M)\eqdef \sum_{i\in I} \langle Me_i,e_i\rangle<+\infty,
\]
where $I\subseteq \N$ and $(e_i)_{i\in I}$ is an orthonormal basis of $\K$. Besides, for $M\in\calL(\K;\H)$, $M^{\star}\in\calL(\H;\K)$ is the adjoint operator of $M$, and for $M\in\calL_2(\K;\H)$, 
\[
\norm{M}_{\mathrm{HS}}\eqdef \sqrt{\tr(MM^{\star})}<+\infty
\] 
is its Hilbert-Schmidt norm (in the finite-dimensional case is equivalent to the Frobenius norm). We denote by $\wlim$ the limit for the weak topology of $\H$. The notation $A: \H\rightrightarrows \H$ means that $A$ is a set-valued operator from $\H$ to $\H$. Consider $f:\H\rightarrow\R$, the sublevel of $f$ at height $r\in\R$ is denoted $[f\leq r]\eqdef \{x\in{\H}: f(x)\leq r\}$. For $1 \leq p \leq +\infty$, $\Lp^p([a,b])$ is the space of measurable functions $g:\R\rightarrow\R$ such that $\int_a^b|g(t)|^p dt<+\infty$, with the usual adaptation when $p = +\infty$. For two functions $f,g:\R\rightarrow\R$ we will denote $f\sim g$ as $t\rightarrow+\infty$, if $\lim_{t\rightarrow+\infty}\frac{f(t)}{g(t)}=1$. 
On the probability space $(\Omega,\calF,\PP)$, $\Lp^p(\Omega;\H)$ denotes the (Bochner) space of $\H$-valued random variables whose $p$-th moment (with respect to the measure $\PP$) is finite. Other notations will be explained when they first appear.\smallskip

Let us recall some important definitions and results from convex analysis; for a comprehensive coverage, we refer the reader to \cite{rocka}.

\smallskip

We denote by $C^{s}(\H)$ the class of $s$-times continuously differentiable functions on $\H$. 
For $L \geq 0$, $C_L^{1,1}(\H) \subset C^{1}(\H)$ is the set of functions on $\H$ whose gradient is $L$-Lipschitz continuous, and $C_L^2(\H)$ is the subset of $C_L^{1,1}(\H)$ whose functions are twice differentiable.\\

The class of $C_L^{1,1}(\H)$ functions enjoys the well-known \textit{descent lemma} which plays a central role in the analysis of optimization dynamics.
\begin{lemma}\label{descent}
Let $f\in C_L^{1,1}(\H)$, then $$f(y)\leq f(x)+\langle \nabla f(x), y-x\rangle+\frac{L}{2}
\norm{y-x}^2,\quad \forall x,y\in\H.$$
\end{lemma}
\begin{corollary}\label{2L}
Let $f\in C_L^{1,1}(\H)$ such that $\argmin f\neq\emptyset$, then 
\begin{equation*}
\norm{\nabla f(x)}^2\leq 2L(f(x)-\min f),\quad \forall x\in \H.\end{equation*}
\end{corollary}

\paragraph*{On stochastic differential equations}
For the necessary notation and preliminaries on stochastic processes, see \cite[Section A.2]{mio2}. Moreover, the existence and uniqueness of a solution of \eqref{ISIHD-S} is discussed in Proposition \ref{existencecor}.\smallskip 

 Let us now present It\^o's formula which plays a central role in the theory of stochastic differential equations:

   \begin{proposition}\label{itos} 
   Consider $(X,V)$ a solution of \eqref{ISIHD-S} and $W$ a $\K-$valued cylindrical Brownian motion, let $\phi: [t_0,+\infty[\times\H\times\H\rightarrow\mathbb{R}$ be such that $\phi(\cdot,x,v)\in C^1([t_0,+\infty[)$ for every $x,v\in\H$, $\phi(t,\cdot,\cdot)\in C^2(\H\times \H)$ for every $t\geq t_0$. Then the process $$Y(t)=\phi(t,X(t),V(t)),$$ is an Itô Process, such that for all $t\geq t_0$
\begin{equation}
\begin{split}
    Y(t)&=Y(t_0)+\int_{t_0}^{t} \frac{\partial \phi}{\partial t}(s,X(s),V(s)) ds
    +\int_{t_0}^{t} \langle  \nabla_x\phi(s,X(s),V(s)),V(s) \rangle ds\\
    &-\int_{t_0}^{t} \langle  \nabla_v\phi(s,X(s),V(s)),\gamma(s)V(s)+\nabla f(X(s)+\beta(s)V(s)) \rangle ds\\
    &+\int_{t_0}^{t}\langle \sigma^{\star}(s,X(s)+\beta(s)V(s))\nabla_v \phi(s,X(s),V(s)) ,dW(s)\rangle\\
    &+\frac{1}{2}\int_{t_0}^{t} \tr[\sigma(s,X(s)+\beta(s)V(s))\sigma^{\star}(s,X(s)+\beta(s)V(s))\nabla_v^2\phi(s,X(s),V(s))]ds,
    \end{split}
\end{equation}
where $\nabla_v^2$ is the Hessian with respect to the double differentiation of $v$ and $\sigma^{\star}$ is the adjoint operator of $\sigma$. Moreover, if for all $T>t_0$ $$\mathbb{E}\left(\int_{t_0}^T \Vert \sigma^{\star}(s,X(s)+\beta(s)V(s))\nabla_v \phi(s,X(s),V(s))\Vert^2 ds\right)<+\infty,$$
then  $\int_{t_0}^{t}\langle \sigma^{\star}(s,X(s)+\beta(s)V(s))\nabla_v \phi(s,X(s),V(s)) ,dW(s)\rangle$ is a square-integrable continuous martingale and
 
\begin{equation}
\mathbb{E}\left(\int_{t_0}^{t}\langle \sigma^{\star}(s,X(s)+\beta(s)V(s))\nabla_v \phi(s,X(s),V(s)) ,dW(s)\rangle\right)=0
\end{equation}
\end{proposition}

\begin{proof}
\tcb{
     Let $Z\in \H\times\H, \Psi:\R_+\times \H\times\H\rightarrow\H\times H,$ and $\Phi(t,Z):\R_+\times \H\times\H\rightarrow\mathcal{M}_{2\times 2}(\mathcal{L}_2(\K;\H))$ such that $$Z=\begin{pmatrix}
        X\\
        V
    \end{pmatrix},
    \Psi(t,Z)=\begin{pmatrix}
        V\\-\gamma(t)V-\nabla f(X+\beta(t)V)
    \end{pmatrix}, \Phi(t,Z)=\begin{pmatrix}
        0 & 0\\
        0 & \sigma(t,X+\beta(t)V)
    \end{pmatrix}.$$
    Let also $W(t)=(W_1(t),W_2(t))$, where $W_1,W_2$ are two independent $\K-$valued cylindrical Brownian motions. Then, we can write \eqref{ISIHD-S} as
    \begin{equation}
\begin{cases}
    Z(t)&=\int_{t_0}^t \Psi(t,Z(t))dt+\int_{t_0}^t \Phi(t,Z(t))dW(t).\\
    Z(t_0)&=(X_0,V_0).
    \end{cases}
    \end{equation}
}
\tcb{
The hypotheses let us use \cite[Section~2.3.2]{infinite} to obtain \begin{align*}
    Y(t)=\phi(t,Z(t))&=Y(t_0)+\int_{t_0}^t \frac{\partial \phi}{\partial t}(s,Z(s))ds+\int_{t_0}^t \langle \nabla_z \phi(s,Z(s)), \Psi(s,Z(s)) \rangle_{\H\times \H}ds\\
    &+\int_{t_0}^t\langle \nabla_z\phi(s,Z(s)) ,\Phi(s,Z(s))dW(s)\rangle_{\H\times\H}\\
    &+\frac{1}{2}\int_{t_0}^t \tr[\nabla_{z}^2\phi(s,Z(s))\Phi(s,Z(s))\Phi^\star(s,Z(s))]ds,
\end{align*}
where $\nabla_z$ and $\nabla_z^2$ denote the gradient and Hessian matrix with respect to the variable $z$, corresponding to first- and second-order derivatives, respectively. We obtain the desired result by replacing the actual values of $\Psi$ and $\Phi$ and using the usual inner product of $\H\times \H$: \[
\langle (x_1, x_2), (y_1, y_2) \rangle_{\H \times \H} \eqdef \langle x_1, y_1 \rangle_\H + \langle x_2, y_2 \rangle_\H.
\]   
}
\end{proof}

\section{General \texorpdfstring{$\gamma$ and $\beta$}{}}\label{sec:nesterov}
In this section, we will develop a Lyapunov analysis based on \cite{hessianpert} to study almost sure, and in expectation properties of the dynamic \eqref{ISIHD-S}, when the parameters $\gamma$ and $\beta$ are general functions. This will allow to go much further and consider parameters not covered in \cite{mio3} which exploits the relationship between first-order and second-order systems. We will also apply our results to two special cases: (i) $\gamma$ is a differentiable, decreasing and vanishing function, with vanishing derivative, and $\beta$ is a positive constant; and (ii) $\gamma(t)=\frac{\alpha}{t}$, and $\beta(t)=\gamma_0+\frac{\beta}{t}$ (with $\gamma_0,\beta >0$). These cases are again are not covered by  results in \cite{mio3}.\smallskip

Recall that our focus in this paper is on an optimization perspective, and as we argued in the introduction, we will study the long time behavior of \eqref{ISIHD-S} as the diffusion term vanishes when $t \to +\infty$. Therefore, throughout the paper, we assume that the diffusion (volatility) term $\sigma$ satisfies:
\begin{equation*}
\tag{$\mathrm{H}_\sigma$}\label{H}
\begin{cases}
\sup_{t \geq t_0,x \in \H} \Vert\sigma(t,x)\Vert_{\HS}<+\infty, \\
\Vert\sigma(t,x')-\sigma(t,x)\Vert_{\HS}\leq l_0\norm{x'-x},
\end{cases}
\end{equation*}   
for some $l_0>0$ and for all  $t\geq t_0, x, x'\in \H$. The Lipschitz continuity assumption is mild and required to ensure the well-posedness of \eqref{ISIHD-S}.
\begin{remark}\label{sigma*}
Under the hypothesis \eqref{H} we have that there exists $\sigma_*^2>0$ such that 
$$
\Vert\sigma(t,x)\Vert_{\HS}^2=\tr[\Sigma(t,x)]\leq \sigma_*^2,\quad \forall t\geq t_0,\forall x\in\H,$$
where $\Sigma\eqdef \sigma\sigma^{\star}$.
Let us also define $\sigma_{\infty}:[t_0,+\infty[\rightarrow \R_+$ as: $\sigma_{\infty}(t)\eqdef \sup_{x\in\H}\Vert\sigma(t,x)\Vert_{\HS}.$
\end{remark}

\noindent Now, we follow with the hypotheses we will require over the damping parameters.

\smallskip

For $t_0> 0$, let $\gamma:[t_0,+\infty[\rightarrow \R_+$ be a viscous damping, we denote \begin{equation}
    p(t)\eqdef \exp\left(\int_{t_0}^t \gamma(s)ds\right).
    \end{equation}
Besides, if 
\begin{equation}\tag{$\mathrm{H}_\gamma$}\label{H1}
    \int_{t_0}^{+\infty}\frac{ds}{p(s)}<+\infty,
\end{equation}
we define $\Gamma:[t_0,+\infty[\rightarrow\R_+$ by \begin{equation}\label{defgamma}
    \Gamma(t)\eqdef p(t)\int_t^{+\infty}\frac{ds}{p(s)}.
\end{equation}

\begin{remark}
Let us notice that $\Gamma$ satisfies the relation $$\Gamma'=\gamma\Gamma-1.$$
\end{remark}

For $t_0> 0$, let $\beta:[t_0,+\infty[\rightarrow \R_+$ be a geometric damping that we will assume to be a differentiable function. We will occasionally need to impose the additional assumption that there exists $c_1,c_2>0$, and $t_1>t_0$ such that for every $t\geq t_1$:  
\begin{equation*}\tag{$\mathrm{H}_\beta$}\label{H2}
    \begin{cases}
    \beta(t)&\leq c_1;  \\
     \Big|\frac{\beta'(t)-\gamma(t)\beta(t)+1}{\beta(t)}\Big|&\leq c_2.
\end{cases}
\end{equation*}

\smallskip

 We recall also that $\calS\eqdef \argmin(f)$.

\subsection{Reformulation of \texorpdfstring{\eqref{ISIHD-S}}{}}
The formulation of the dynamic \eqref{ISIHD-S} is known as the Hamiltonian formulation. 
However, it is not the only one. In the deterministic case, an alternative equivalent and more flexible first-order reformulation of \eqref{ISIHD} was proposed in \cite{hessianpert}. The motivation there was that this equivalent reformulation can handle the case where $f$ is non-smooth. Although we will not consider the non-smooth case here, we will still extend and use that equivalent reformulation to the stochastic case.\smallskip

Consider the dynamic \eqref{ISIHD-S}, and let us define the auxiliary variable 
$$Y(t)=X(t)+\beta(t)V(t), \quad t>t_0 .$$
We have that \begin{align*}
    dY(t)&=dX(t)+\beta'(t)V(t)+\beta(t)dV(t)\\
    &=-\beta(t)\nabla f(Y(t))dt-(\beta'(t)-\gamma(t)\beta(t)+1)\left(\frac{X(t)-Y(t)}{\beta(t)}\right)dt+\beta(t)\sigma(t,Y(t))dW(t).
\end{align*}
So we can reformulate \eqref{ISIHD-S} in terms of $X,Y$ in the following way:
\begin{align} \label{refor}\tag{$\mathrm{ISIHD-S_R}$}
\begin{cases}
dX(t)&=-\left(\frac{X(t)-Y(t)}{\beta(t)}\right)dt,\quad t>t_0, \nonumber\\
dY(t)&=-\beta(t)\nabla f(Y(t))dt-(\beta'(t)-\gamma(t)\beta(t)+1)\left(\frac{X(t)-Y(t)}{\beta(t)}\right)dt+\beta(t)\sigma(t,Y(t))dW(t),\quad t>t_0,\\
X(t_0)&=X_0,\quad Y(t_0)=X_0+\beta(t_0)V_0, \nonumber
\end{cases}
\end{align}
where the subscript 'R' indicates that this is a reformulation. Moreover, we can reformulate \eqref{refor} in the product space $\H\times\H$ by setting $Z(t)=(X(t),Y(t))\in\H\times\H$, and thus \eqref{refor} can be equivalently written as \begin{equation}\label{refor2}\begin{cases}
    dZ(t)&=-\beta(t)\nabla \mathcal{G}(Z(t))dt-\mathcal{D}(t,Z(t))dt+\hat{\sigma}(t,Z(t))dW(t),\quad t>t_0,\\
    Z(t_0)&=(X_0,X_0+\beta(t_0)V_0),\end{cases}
\end{equation}
where $\mathcal{G}:\H\times\H\rightarrow\R$ is the convex function defined as $\mathcal{G}(Z)=f(Y)$, and the time-dependent operator $\mathcal{D}:[t_0,+\infty[\times\H\times\H\rightarrow\H\times\H$ is given by \begin{equation}
    \mathcal{D}(t,Z)=\left(\frac{1}{\beta(t)}(X-Y),\frac{\beta'(t)-\gamma(t)\beta(t)+1}{\beta(t)}(X-Y)\right),
\end{equation}
and the stochastic noise $\hat{\sigma}\in\mathcal{M}_{2\times 2}(\calL_2(\K;\H))$ defined by $\hat{\sigma}(t,Z)=\begin{pmatrix}0 &0 \\
0 & \beta(t)\sigma(t,Y)\end{pmatrix}$, and 
$W(t)=(W_1(t),W_2(t))$, where $W_1,W_2$ are two independent $\K-$valued cylindrical Brownian motions.

\subsection{Fast convergence properties: convex case}
To obtain properties in almost sure sense and in expectation of \eqref{ISIHD-S}, we are going to adapt the Lyapunov analysis shown on \cite{hessianpert} for the dynamic \eqref{ISIHD}. \smallskip

To that purpose, let us consider $t_1>t_0$, $\gamma,\beta:[t_0,+\infty[\rightarrow\R_+$ be fixed functions and let $a,b,c,d:[t_0,+\infty[\rightarrow\R$ be differentiable functions (on $]t_0,+\infty[$) satisfying the following system for all $t>t_1$:
\begin{equation}\label{systemabcd}\tag{$\mathrm{S}_{a,b,c,d}$}
    \begin{cases}
    a'(t)-b(t)c(t)&\leq0\\
    -a(t)\beta(t)&\leq 0\\
    -a(t)\gamma(t)\beta(t)+a(t)\beta'(t)+a(t)-c(t)^2+b(t)c(t)\beta(t)&=0\\
    b'(t)b(t)+\frac{d'(t)}{2}&\leq 0\\
    b'(t)c(t)+b(t)(b(t)+c'(t)-c(t)\gamma(t))+d(t)&=0\\
    c(t)(b(t)+c'(t)-c(t)\gamma(t))&\leq 0.
    \end{cases}
\end{equation}
Given $x^{\star}\in\mathcal{S}$, we consider \begin{equation}\label{liapunov}
    \mathcal{E}(t,x,v)=a(t)(f(x+\beta(t)v)-\min(f))+\frac{1}{2}\Vert b(t)(x-x^{\star})+c(t)v\Vert^2+\frac{d(t)}{2}\Vert x-x^{\star}\Vert^2.
\end{equation}
\begin{remark}
It was shown in \cite[Section~3.1]{alecsa} and \cite[Lemma~1]{hessianpert} that energy function $\mathcal{E}$ with $a,b,c,d$ satisfying the system \eqref{systemabcd} is a Lyapunov function for \eqref{ISIHD} when $\gamma(t)=\frac{\alpha}{t}$ (with $\alpha>3$) and $\beta(t)=\gamma_0+\frac{\beta}{t}$ (with $\gamma_0,\beta\geq 0$), hence, useful to obtain convergence guarantees of that dynamic. We will see that the same system \eqref{systemabcd} also covers the case of general coefficients $\gamma$ and $\beta$, hence providing insights on the convergence properties of \eqref{ISIHD-S} when one can find the corresponding functions $a,b,c,d$. 
\end{remark}
In the following proposition, we state an abstract integral bound, as well as almost sure and in expectation convergence properties for \eqref{ISIHD-S}.
\begin{proposition}\label{abcd}
Assume that $f,\sigma$ satisfy \eqref{H0} and \eqref{H}, respectively. Let $\nu\geq 2$, and consider the dynamic \eqref{ISIHD-S} with initial data $X_0,V_0\in\Lp^{\nu}(\Omega;\H)$. Consider also $\gamma, \beta$ from \eqref{ISIHD-S} satisfying \eqref{H1} and \eqref{H2}. Then, there exists a unique solution $(X,V)\in S_{\H\times \H}^{\nu}[t_0]$ of \eqref{ISIHD-S}. \tcb{Moreover, if there exist functions $a,b,c,d$ satisfying \eqref{systemabcd}, and} $t\mapsto m(t)\sigma_{\infty}^2(t)\in\Lp^1([t_0,+\infty[)$, where $m(t)\eqdef\max\{1,a(t),c^2(t)\}$, then the following statements hold: 
\begin{enumerate}[label=(\roman*)]
    \item 
    If $b(t)c(t)-a'(t)=\mathcal{O}(c(t)(\gamma(t)c(t)-c'(t)-b(t)))$, then $$\int_{t_0}^{+\infty}(b(s)c(s)-a'(s))(f(X(s))-\min f+\Vert V(s)\Vert^2)ds<+\infty \quad a.s..$$ 
\item  If there exists $\eta>0,\hat{t}>t_0$ such that $$\eta\leq c(t)(\gamma(t)c(t)-c'(t)-b(t)),\quad \eta\leq a(t)\beta(t),\quad \gamma(t)\leq \eta,\quad \forall t>\hat{t},$$ then $\lim_{t\rightarrow +\infty}\Vert V(t)\Vert=0$ a.s., $\lim_{t\rightarrow +\infty}\Vert \nabla f(X(t)+\beta(t)V(t))\Vert=0$ a.s., and 
$\lim_{t\rightarrow +\infty}\Vert \nabla f(X(t))\Vert=0$ a.s.
\item   If there exists $D>0, \tilde{t}>t_0$ such that $d(t)\geq D$ for $t>\tilde{t}$, then :
$$\mathbb{E}\left(\Vert V(t)\Vert^2\right)=\mathcal{O}\left(\frac{1+b^2(t)}{c^2(t)}\right),$$
and $$ \mathbb{E}\left(f(X(t))-\min f\right)=\mathcal{O}\left(\max\Bigg\{\frac{1}{a(t)},\frac{\beta(t)\sqrt{1+b^2(t)}}{\sqrt{a(t)}c(t)},\frac{\beta^2(t)\left(1+b^2(t)\right)}{c^2(t)}\Bigg\}\right).$$

\end{enumerate}
\end{proposition}
This is a compact version that extracts only the most important points from the more detailed and complete Propositions \ref{abcdcomplete}, \ref{abcdexpcomplete} which are proved in the appendix.

The complete version of the previous proposition (\ie Propositions \ref{abcdcomplete} and \ref{abcdexpcomplete}) generalizes the results poved in \cite{hessianpert} to the stochastic setting. However, they lack practical use if we cannot exhibit $a,b,c,d$ functions that satisfy \eqref{systemabcd}. Although we are not able to solve this system in general, in Corollaries \ref{corabcdd} and \ref{cor1} we will specify some particular cases for $\gamma$ and $\beta$ where such functions $a,b,c,d$ can be exhibited to satisfy the system \eqref{systemabcd}.

\smallskip

The following corollary provides a specific case where a solution to the system \eqref{systemabcd} can be exhibited, which was not discussed in \cite{alecsa,hessianpert}. Moreover, we show the implications it has on the stochastic setting.

\begin{corollary}[Decreasing and vanishing $\gamma$, with vanishing $\gamma'$ and positive constant $\beta$]\label{corabcdd}
     Consider the context of Proposition \ref{abcd} in the case where $\beta(t)\equiv\beta>0$, $\gamma$ satisfying \eqref{H1}, such that it is a differentiable, decreasing, and vanishing function, with $\lim_{t\rightarrow +\infty}\gamma'(t)=0$, and satisfying that: 
     \begin{equation} \label{H30}\tag{$\mathrm{H}_\gamma'$}
    \text{there exists } t_2\geq t_0 \text{ and } m<\frac{3}{2} \text{ such that } \gamma(t)\Gamma(t)\leq m \text{ for every } t\geq t_2.
\end{equation}  
Let $b\in ]2(m-1),1[$, then choosing \begin{align*}
     a(t)&=\frac{\Gamma(t)(\Gamma(t)-\beta b)}{1-\beta\gamma(t)},\\
     b(t)&=b,\\
     c(t)&=\Gamma(t),\\
     d(t)&=b(1-b),
 \end{align*}
there exists $\hat{t}>t_0$ such that the system \eqref{systemabcd} is satisfied for every $t\geq \hat{t}$.\smallskip

Given $x^{\star}\in\calS$ and $\sigma_{\infty}$ be such that $t\mapsto \Gamma(t)\sigma_{\infty}(t)\in\Lp^2([t_0,+\infty[)$, then the following statements hold:
 \begin{enumerate}[label=(\roman*)]
      \item $\int_{t_0}^{+\infty} \Gamma(s)\pa{f(X(s))-\min f+\Vert V(s)\Vert^2}ds<+\infty$ a.s..
      \item $\lim_{t\rightarrow +\infty}\Vert\nabla f(X(t))\Vert+\Vert V(t)\Vert=0$ a.s..
    
         \item \label{1gamma2} $\mathbb{E}(f(X(t))-\min f+\Vert V(t)\Vert^2)=\mathcal{O}\left(\frac{1}{\Gamma^2(t)}\right)$.
     \end{enumerate}
\end{corollary}
\begin{remark}\label{remarkimp}
    When $\gamma(t)=\frac{\alpha}{t}$ with $\alpha>3$ and $t\sigma_{\infty}(t)\in\Lp^2([t_0,+\infty[)$, the previous corollary ensures fast convergence of the values, $\ie, \mathcal{O}(t^{-2})$. Besides, by Corollary \ref{atr}, when $\gamma(t)=\frac{\alpha}{t^r}$ with $r\in ]0,1[,\alpha\geq 1-r, $ and $t^r\sigma_{\infty}(t)\in\Lp^2([t_0,+\infty[)$, the previous corollary ensures convergence of the objective at a rate $\mathcal{O}\left(t^{-2r}\right)$. The latter choice indicates that one can require a weaker integrability condition on the noise, compared to the case $\gamma(t)=\frac{\alpha}{t}$ ($\alpha>3$), but at the price of a slower convergence rate.
\end{remark}
\begin{proof}
    We start by noticing that since $\gamma$ is decreasing, by \cite[Corollary 2.3]{AC1} we have that $\Gamma(t)$ is increasing and $\gamma(t)\Gamma(t)\geq 1$, for every $t\geq t_0$. Also, it is direct that with a fixed $\beta>0$ we satisfy \eqref{H2}.
    
    \smallskip
    
    Letting $b\in ]2(m-1),1[$ and $t_1> t_0$ such that $\beta\leq \frac{1}{\gamma(t_1)}$, this $t_1$ exists since $t\mapsto \frac{1}{\gamma(t)}$ is an increasing function. We choose $c(t)=\Gamma(t)$, by the fifth equation of \eqref{systemabcd}, we get that $d=b(1-b)$, and the fourth equation is trivial.  The third equation implies that $a(t)=\frac{\Gamma(t)(\Gamma(t)-b\beta)}{1-\beta\gamma(t)}$ and the choice of $\beta$ implies that the second equation is satisfied for $t\geq t_1$, since $\beta\leq \frac{1}{\gamma(t_1)}\leq \frac{1}{\gamma(t)}\leq\Gamma(t)$ for every $t>t_1$. By the definition of $c(t)$ and the fact that $b<1$, we directly have that the sixth equation also holds. We just need to check the first equation, to do so, we can see that this equation is equivalent to $$\frac{\Gamma'(t)(2\Gamma(t)-\beta b)(1-\beta\gamma(t))+\beta\Gamma(t)(\Gamma(t)-\beta b)\gamma'(t)}{(1-\beta\gamma(t))^2}\leq b\Gamma(t),$$
    \tcb{by multiplying by $(1-\beta\gamma(t))^2$ at both sides and developing the terms, we get the previous inequality is equivalent to 
    \begin{multline}\label{topr}
    2\Gamma(t)\Gamma'(t)-\beta b\Gamma'(t) -2\beta\gamma(t)\Gamma(t)\Gamma'(t)+b\beta^2\gamma(t)\Gamma'(t)+\beta\Gamma^2(t)\gamma'(t)-b\beta^2\Gamma(t)\gamma'(t) \\
    \leq b\Gamma(t)-2b\beta \gamma(t)\Gamma(t)+b\beta^2\gamma^2\tcb{(t)}\Gamma(t).
    \end{multline}
     \tcb{By \eqref{H30}, there exists $m>\frac{3}{2}$ and $t_2>t_0$ such that $1\leq \gamma(t)\Gamma(t)\leq m$ and $0\leq \Gamma'(t)\leq m-1$} for every $t\geq t_2$.  Since the terms $$-2\beta\gamma(t)\Gamma(t)\Gamma'(t),\quad -\beta b\Gamma'(t),\quad \beta\Gamma^2(t)\gamma'(t)$$ are negative, and $b\beta^2\gamma^2(t)\Gamma(t)$ is positive, thus if we could prove there exists $t_3>\max\{t_0,t_1,t_2\}$ such that for every $t\geq t_3$
    \begin{equation}\label{topr1}
2\Gamma(t)\Gamma'(t)+b\beta^2\gamma(t)\Gamma'(t)-b\beta^2\Gamma(t)\gamma'(t)  \leq b\Gamma(t)-2b\beta \gamma(t)\Gamma(t),
    \end{equation}
    we could conclude that \eqref{topr} holds for every $t\geq t_3$. Rearranging the terms in \eqref{topr1}, we get \begin{equation}\label{topr2}
b\beta^2\gamma(t)\Gamma'(t) +2b\beta \gamma(t)\Gamma(t) \leq \Gamma(t)[b\beta^2\gamma'(t)+b-2\Gamma'(t)],
    \end{equation} and we note that the terms $ b\beta^2\gamma(t)\Gamma'(t), 2b\beta \gamma(t)\Gamma(t)$ are upper bounded by a constant. Since $\lim_{t\rightarrow+\infty}\gamma(t)=0$, we get that $\lim_{t\rightarrow+\infty}\Gamma(t)=+\infty$, therefore if we could prove that there exists $t_3> \max\{t_0,t_1,t_2\}$ such that $$-b\beta^2\gamma'(t)<\delta<b-2(m-1)\leq b-2\Gamma'(t)$$
    for $t\geq t_3$, this would imply that there exists $\hat{t}\geq t_3$ such that \eqref{topr2} holds, which in turn would imply that \eqref{topr} holds for every $t\geq \hat{t}$}. In fact, we see that the previous inequality holds for $t$ large enough (\ie there exists such a $t_3$) since $\lim_{t\rightarrow +\infty}-\gamma'(t)=0$ and the fact that \tcb{$0<b-2(m-1)$ implies that there exists $\delta$ such that $0<\delta<b-2(m-1)$}. Thus, we have checked that the proposed $a,b,c,d$ satisfy the system \eqref{systemabcd} for $t>\hat{t}$.

    \smallskip

    \tcb{The rest of the proof is direct by replacing the specified $a,b,c,d,\gamma,\beta$ functions in Proposition~\ref{abcd}, and the fact that for $t$ large enough, $b\Gamma(t)-a'(t)\geq (b-2(m-1)\tcb{-\delta})\Gamma(t)-C_b$ for some $C_b>0$, that $\lim_{t\rightarrow +\infty}\Gamma(t)=+\infty$, and also that $a(t)=\mathcal{O}(\Gamma^2(t)), \Gamma^2(t)=\mathcal{O}(a(t))$. Since $\lim_{t\rightarrow +\infty}\Gamma(t)=+\infty, \lim_{t\rightarrow +\infty}a(t)=+\infty$, and $\lim_{t\rightarrow +\infty} \gamma(t)=0$, the parameter $\eta>0$ in Proposition~\ref{abcd} can be chosen arbitrarily.}
    
\end{proof}

The following result gives us another case in which we can satisfy the system \eqref{systemabcd}. This generalizes to the stochastic setting the results presented in \cite[Section~3.1]{alecsa} and \cite[Lemma~1]{hessianpert}. Besides, it ensures fast convergence of the values whenever $t\mapsto t\sigma_{\infty}(t)\in\Lp^2([t_0,+\infty[)$.

\begin{corollary}[$\gamma(t)=\frac{\alpha}{t}$ and $\beta(t)=\gamma_0+\frac{\beta}{t}$]\label{cor1}
 Consider the context of Proposition \ref{abcd} in the case where $\gamma(t)=\frac{\alpha}{t}$ and $\beta(t)=\gamma_0+\frac{\beta}{t}$, where $\alpha>3,\gamma_0>0,\beta\geq 0$. Then choosing \begin{align*}
     a(t)&=t^2\pa{1+\frac{(\alpha-b)\gamma_0 t-\beta(\alpha+1-b)}{t^2-\alpha\gamma_0 t-\beta(\alpha+1)}},\\
     b(t)&=b\in(2,\alpha-1),\\
     c(t)&=t,\\
     d(t)&=b(\alpha-1-b),
 \end{align*}
 the system \eqref{systemabcd} is satisfied.\smallskip
 
 Given $x^{\star}\in\calS$ and $\sigma_{\infty}$ be such that $t\mapsto t\sigma_{\infty}(t)\in\Lp^2([t_0,+\infty[)$, then the following statements hold:
 \begin{enumerate}[label=(\roman*)]
      \item $\int_{t_0}^{+\infty} s\pa{f(X(s))-\min f+\Vert V(s)\Vert^2}ds<+\infty$ a.s..
      \item $\lim_{t\rightarrow +\infty}\Vert\nabla f(X(t))\Vert+\Vert V(t)\Vert=0$ a.s.. 
    
         \item $\mathbb{E}(f(X(t))-\min f+\Vert V(t)\Vert^2)=\mathcal{O}\left(\frac{1}{t^2}\right)$.
     \end{enumerate}

\end{corollary}
\begin{proof}
Direct from replacing the specified $a,b,c,d,\gamma,\beta$ functions in Proposition \ref{abcd}, and the fact that for $t$ large enough $bt-a'(t)\geq \frac{(\alpha-3)t}{2}$, and also $a(t)\geq t^2, 0<\gamma_0<\beta(t)\leq \gamma_0+\frac{\beta}{t_0}.$
\end{proof}
\begin{remark}
    We can use the choices for $a,b,c,d$ presented in Corollaries \ref{corabcdd} and \ref{cor1} in Propositions \ref{abcdcomplete} and \ref{abcdexpcomplete} to obtain additional integral bounds, almost sure and in expectation properties of \eqref{ISIHD-S}. We leave this to the reader.
\end{remark}
\subsection{Strongly convex case}
In the following theorem, we consider the case where the objective function is strongly convex and we present a choice of parameters $\gamma$ and $\beta$ to obtain a fast linear convergence rate when the noise vanishes at a proper rate.
\begin{theorem}
Assume that $f:\H\rightarrow\R$ satisfies \eqref{H0}, and is $\mu$-strongly convex, $\mu > 0$, and denote $x^{\star}$ its unique minimizer. Suppose also that $\sigma$ obeys \eqref{H}. Let $\nu\geq 2$, consider the dynamic \eqref{ISIHD-S} with initial data $X_0\in\Lp^{\nu}(\Omega;\H)$. Consider also $\gamma\equiv 2\sqrt{\mu}$, and a constant $\beta$ such that $0\leq\beta\leq\frac{1}{2\sqrt{\mu}}$. Moreover, suppose that $\sigma_{\infty}$ is a non-increasing function such that $\sigma_{\infty}\in\Lp^2([t_0,+\infty[)$. Define the function $\mathcal{E}:[t_0,+\infty[\times\H\times\H\rightarrow\R_+$ as 
    \[
    \mathcal{E}(t,x,v)\eqdef f(x+\beta v)-\min f+\frac{1}{2}\Vert \sqrt{\mu}(x-x^{\star})+v\Vert^2. 
    \]
    Then, \eqref{ISIHD-S} has a unique solution $(X,V)\in S_{\H\times \H}^{\nu}[t_0]$. In addition, there exists positive constants $M_1,M_2$ such that 
    \[
    \mathbb{E}[\mathcal{E}(t,X(t),V(t))]\leq \mathcal{E}(t_0,X_0,V_0)e^{-\frac{\sqrt{\mu}}{2}(t-t_0)}+M_1e^{-\frac{\sqrt{\mu}}{4}(t-t_0)}+M_2\sigma_{\infty}^2\left(\frac{t_0+t}{2}\right),\quad\forall t>t_0.
    \]
    Let $\Theta:[t_0,+\infty[\rightarrow\R_+$ defined as $\Theta(t)\eqdef\max\{e^{-\frac{\sqrt{\mu}}{4}(t-t_0)},\sigma_{\infty}^2\left(\frac{t+t_0}{2}\right)\}$. Consequently, 
    \begin{align*}
        \mathbb{E}(f(X(t))-\min f)&=\mathcal{O}(\Theta(t)),\\
        \mathbb{E}(\Vert X(t))-x^{\star}\Vert^2)&=\mathcal{O}(\Theta(t)),\\
        \mathbb{E}(\Vert V(t)\Vert^2)&=\mathcal{O}(\Theta(t)),\\
        \mathbb{E}(\Vert \nabla f(X(t))\Vert^2)&=\mathcal{O}(\Theta(t)).
    \end{align*}
\end{theorem}
\begin{proof}
     Using It\^o's formula with $\mathcal{E}$, taking expectation and denoting $E(t)\eqdef \mathbb{E}(\mathcal{E}(t,X(t),V(t)))$, we have \begin{align*}
         E(t)\leq E(t_0)-\int_{t_0}^t \frac{\sqrt{\mu}}{2}E(s)ds-\int_{t_0}^t C(s)ds+(L\beta^2+1)\int_{t_0}^t \sigma_{\infty}^2(s) ds, 
     \end{align*}
     where \begin{align*}
         C(t)&\eqdef \beta\Vert\nabla f(X(t)+\beta V(t))\Vert^2+\beta\sqrt{\mu}\langle \nabla f(X(t)+\beta V(t)),V(t) \rangle+\frac{\sqrt{\mu}}{2}(\beta^2\mu+1)\Vert V(t)\Vert^2\\
         &+\beta\mu\sqrt{\mu}\langle X(t)-x^{\star},V(t) \rangle +\frac{\mu\sqrt{\mu}}{4}\Vert X(t)-x^{\star}+\beta V(t)\Vert^2.
     \end{align*}
     It was proved in \cite[Theorem 4.2]{hessianpert} that under the condition $0\leq\beta\leq\frac{1}{2\sqrt{\mu}}$ we obtain that $C(t)$ is a non-negative function. Therefore, we can write the following $$E(t)\leq E(t_0)-\int_{t_0}^t \frac{\sqrt{\mu}}{2}E(s)ds+(L\beta^2+1)\int_{t_0}^t \sigma_{\infty}^2(s) ds. $$
     We continue by using \cite[Lemma A.2]{mio2}, to do this, we need to solve the following Cauchy problem: 
     \begin{align*}
    \centering
     \begin{cases}
         Y'(t)&=-\frac{\sqrt{\mu}}{2}Y(t)+(L\beta^2+1)\sigma_{\infty}^2(t)\\
         Y(t_0)&=\mathcal{E}(t_0,X_0,V_0).
     \end{cases}
     \end{align*}
     
    Using the integrating factor method, we deduce that for all $t\geq t_0$: \begin{align*}
        Y(t)&=Y(t_0)e^{\frac{\sqrt{\mu}}{2}(t_0-t)}+(L\beta^2+1)e^{-\frac{\sqrt{\mu}}{2}t}\int_{t_0}^t e^{\frac{\sqrt{\mu}}{2}s}\sigma_{\infty}^2(s)ds\\
        &\leq Y(t_0)e^{\frac{\sqrt{\mu}}{2}(t_0-t)}+(L\beta^2+1)e^{-\frac{\sqrt{\mu}}{2}t}\pa{\int_{t_0}^{\frac{t_0+t}{2}} e^{\frac{\sqrt{\mu}}{2}s}\sigma_{\infty}^2(s)ds+\int_{\frac{t_0+t}{2}}^t e^{\frac{\sqrt{\mu}}{2}s}\sigma_{\infty}^2(s)ds}\\
        &\leq Y(t_0)e^{\frac{\sqrt{\mu}}{2}(t_0-t)}+\tcb{\frac{2(L\beta^2+1)}{\sqrt{\mu}}}\sigma_{\infty}^2\pa{\frac{t_0+t}{2}}+\tcb{\frac{2(L\beta^2+1)}{\sqrt{\mu}}}\sigma_{\infty}^2(t_0) e^{\frac{\sqrt{\mu}}{4}(t_0-t)}\\
        &=\mathcal{O}(\Theta(t)).
    \end{align*}
     By \cite[Lemma A.2]{mio2}, we conclude that $E(t)=\mathcal{O}(\Theta(t))$, immediately we observe that \begin{align*}
         \mathbb{E}(f(X(t)+\beta V(t))-\min f)&=\mathcal{O}(\Theta(t))\\
         \mathbb{E}(\Vert\sqrt{\mu}(X(t)-x^{\star})+V(t)\Vert^2)&=\mathcal{O}(\Theta(t))
     \end{align*}
     By the strong convexity of $f$, we have that $\mathbb{E}(\Vert X(t)-x^{\star}+\beta V(t)\Vert^2)=\mathcal{O}(\Theta(t))$, since $\beta\neq\frac{1}{\sqrt{\mu}}$ ($\beta\leq \frac{1}{2\sqrt{\mu}}$), then $\mathbb{E}(\Vert X(t)-x^{\star}\Vert^2)=\mathbb{E}(\Vert V(t)\Vert^2)=\mathcal{O}(\Theta(t))$, on the other hand, using Lemma \ref{descent} and Lemma \ref{2L}, $\mathbb{E}(f(X(t))-f(X(t)+\beta V(t)))=\mathcal{O}(\Theta(t))$, thus, $$\mathbb{E}(f(X(t))-\min f )=\mathbb{E}(\Vert \nabla f(X(t))\Vert^2)=\mathcal{O}(\Theta(t)).$$
\end{proof}

\section{General \texorpdfstring{$\gamma$ and $\beta\equiv 0$}{}}\label{sec:nesterov0}
In this section we are going to study properties of the dynamic \eqref{ISIHD-S} in expectation and in almost sure sense, when the parameter $\gamma$ is a general function and $\beta\equiv 0$. The noiseless case and under deterministic noise is well documented in \cite{cabot}.\smallskip

Consider the dynamic \eqref{ISIHD-S} when $\beta\equiv 0$. This dynamic will be a stochastic version of the Hamiltonian formulation of \eqref{AVD} and it will be described by: 
\begin{align} \label{NAG}\tag{$\mathrm{IGS}_{\gamma}-\mathrm{S}$}
\begin{cases}
dX(t)&=V(t)dt,\hspace{0.1cm} t>t_0, \nonumber\\
dV(t)&=-\gamma(t)V(t)dt-\nabla f(X(t))dt+\sigma(t,X(t))dW(t),\hspace{0.1cm} t>t_0,\\
X(t_0)&=X_0,\quad V(t_0)=V_0. \nonumber
\end{cases}
\end{align}
The main motivation for a separate analysis is that, in Section \ref{sec:nesterov} we consider hypothesis \eqref{H2} to establish the existence and uniqueness of a solution, from which, the rest of the results follow. This hypothesis is incompatible with the case $\beta\equiv 0$.

\smallskip

We will demonstrate almost sure convergence of the velocity to zero and of the objective to its minimum value, under assumptions that are satisfied for $\gamma(t)=\frac{\alpha}{t^r}$, with $r\in [0,1], \alpha\geq 1-r$. Additionally, we will show that for this particular choice of $\beta$, we can obtain almost sure (weak) convergence of the trajectory.

\subsection{Minimization properties}
Let us define for $c>0$, $$\lambda_c(t)=\frac{p(t)}{c+\int_{t_0}^t p(s)ds}.$$
We can deduce that $\lambda_c'+\lambda_c^2-\gamma\lambda_c=0$, besides, since $p\notin\Lp^1([t_0,+\infty[)$, then $\lambda_c\notin\Lp^1([t_0,+\infty[)$. 
\begin{theorem}\label{minimization}
Assume that $f$ and $\sigma$ satisfy assumptions \eqref{H0} and \eqref{H}, respectively. Let $\nu\geq 2$, and consider the dynamic \eqref{NAG} with initial data $X_0,V_0\in\Lp^{\nu}(\Omega;\H)$. Then, there exists a unique solution $(X,V)\in S_{\H\times\H}^{\nu}[t_0]$ of \eqref{NAG}. Additionally, if $\sigma_{\infty}\in\Lp^2([t_0,+\infty[)$, then $$\int_{t_0}^{+\infty} \gamma(s)\Vert V(s)\Vert^2 ds<+\infty \quad a.s..$$ 
Moreover, suppose that
\begin{equation} \label{Ha}\tag{$\mathrm{H}_a$}
    \text{there exists } \hat{t}\geq t_0,\text{ and } c>0 \text{ such that } \gamma(t)\leq\lambda_c(t)\hspace{0.1cm} \forall t\geq \hat{t},
\end{equation} 
  and 
  \begin{equation}\label{Hb}\tag{$\mathrm{H}_b$}
      \int_{t_0}^{+\infty} \lambda_c(s)\Vert V(s)\Vert^2 ds<+\infty \quad a.s.. 
  \end{equation}
  Then the following properties are satisfied: 
  \begin{enumerate}[label=(\roman*)]
    \item $\int_{t_0}^{+\infty} \lambda_c(s)(f(X(s)-\min f) ds<+\infty$ a.s..
    \item $\lim_{t\rightarrow +\infty} \Vert V(t)\Vert=0$ a.s. and $\lim_{t\rightarrow +\infty} f(X(t))-\min f=0$ a.s..
\end{enumerate}
\end{theorem}
\begin{proof}
The existence and uniqueness of a solution of \eqref{NAG} is a direct consequence of \cite[Theorem 3.3]{mio2} in the product space $\H\times \H$. \smallskip

Let $x^{\star}\in\calS$ and $\phi_0:(x,v)\mapsto \R$ defined by $\phi_0(x,v)=f(x)-\min f+\frac{\Vert v\Vert^2}{2}$, by It\^o's formula and Theorem~\ref{impp} we obtain that $\int_{t_0}^{+\infty} \gamma(s)\Vert V(s)\Vert^2 ds<+\infty$ a.s.. Moreover, if we assume the hypotheses \eqref{Ha} and \eqref{Hb}, then:
\begin{enumerate}[label=(\roman*)]
    \item Let $x^{\star}\in\calS$ and $\phi:(t,x,v)\mapsto \R$ defined by $\phi(t,x,v)=\frac{\Vert \lambda_c(t)(x-x^{\star})+v\Vert^2}{2}+(f(x)-\min f)$. Let $\hat{t}$ defined in the statement, by It\^o's formula from $\hat{t}$ to $t$, we have \begin{align*}
    &f(X(t))-\min f+\frac{\Vert \lambda_c(t)(X(t)-x^{\star})+V(t)\Vert^2}{2}=f(X(\hat{t}))-\min f+\frac{\Vert \lambda_c(\tcb{\hat{t}})(X(\hat{t})-x^{\star})+V(\hat{t})\Vert^2}{2}\\&+\int_{\hat{t}}^t\lambda_c(s)\lambda_c'(s)\Vert X(s))-x^{\star}\Vert^2-\gamma(t)\Vert V(s)\Vert^2-\lambda_c(t)\langle \nabla f(X(s)), X(s)-x^{\star} \rangle]ds\\
    &+ \int_{\hat{t}}^t \lambda_c(t)\Vert V(s)\Vert^2+\tr[\Sigma(s,X(s))]ds+\underbrace{\int_{\hat{t}}^t\langle [\lambda_c(s)(X(s)-x^{\star})+V(s)]\sigma^{\star}(s,X(s)), dW(s)\rangle}_{M_t}.
\end{align*}
By the hypotheses, we have that $$\int_{\hat{t}}^{+\infty} \left(\lambda_c(s)\Vert V(s)\Vert^2+\tr[\Sigma(s,X(s))]\right)ds\leq \int_{\hat{t}}^{+\infty} \left(\lambda_c(s)\Vert V(s)\Vert^2+\sigma_{\infty}^2(s)\right)ds<+\infty \quad  a.s..$$
Besides $(M_t)_{t\geq \hat{t}}$ is a continuous martingale. Moreover, by convexity of $f$ and the fact that $\lambda_c'(t)\leq 0$ $\forall t\geq \hat{t}$, \begin{align*}
    &\int_{\hat{t}}^t\lambda_c(s)\lambda_c'(s)\Vert X(s))-x^{\star}\Vert^2-\gamma(t)\Vert V(s)\Vert^2-\lambda_c(t)\langle \nabla f(X(s)), X(s)-x^{\star} \rangle]ds\\
    &\leq -\int_{\hat{t}}^t \lambda_c(s)(f(X(s))-\min f)ds.
\end{align*}
Then, by Theorem~\ref{impp}, \begin{equation}\label{lamf}
     \int_{\hat{t}}^{+\infty}\lambda_c(s)(f(X(s))-\min f)ds<+\infty \quad a.s.,
 \end{equation} and $\lim_{t\rightarrow +\infty}f(X(t))-\min f+\frac{\Vert \lambda_c(t)(X(t)-x^{\star})+V(t)\Vert^2}{2}$ exists a.s.. 
 \item By Lemma \ref{lim0} and \eqref{lamf}, we conclude that $\lim_{t\rightarrow +\infty} \frac{\Vert V(t)\Vert^2}{2}+f(X(t))-\min f=0$ a.s..
\end{enumerate}
\end{proof}
\begin{corollary}
Consider the context of Theorem \ref{minimization} with $\gamma(t)=\frac{\alpha}{t^r}$, where $r\in [0,1]$ and $\alpha> 1-r$. Then \eqref{Ha} and \eqref{Hb} are satisfied and thus the conclusions of Theorem~\ref{minimization} hold.
\end{corollary}
\begin{proof}
\begin{itemize}
\item We will prove the case $r=1$ first, since it is direct, in such case, letting $c=\frac{t_0}{\alpha+1}$ we have that $\lambda_c(t)=\frac{\alpha+1}{t}$, which satisfies \eqref{Ha}, moreover, $\lambda(t)=\frac{\alpha+1}{\alpha}\gamma(t)$, so \eqref{Hb} is also satisfied.
\item Let $r\in ]0,1[$, $c=\frac{\int_{0}^{t_0}e^{\alpha s^{1-r}}ds}{e^{\alpha t_0^{1-r}}}$. Instead of proving $\gamma(t)\leq\lambda_c(t)$, we will prove the equivalent inequality $\frac{1}{t\lambda_c(t)}\leq \frac{1}{t\gamma(t)}$. In fact, by a change of variable we have that (see notation of $I_p$ in Lemma \ref{integral}): \begin{align*}
    \frac{1}{\lambda_c(t)t}&=\frac{(1-r)^{\frac{r}{1-r}}}{\alpha^{\frac{1}{1-r}}}I_{\frac{r}{1-r}}\left(\frac{\alpha}{1-r}t^{1-r}\right),
\end{align*}
Moreover, by the first result of Lemma \ref{integral} we have that $$\frac{1}{t\lambda_c(t)}\leq \left(\frac{1-r}{\alpha}\right)^{\frac{1}{1-r}}\frac{t^{r-1}}{\alpha}\leq\frac{1}{t\gamma(t)}.$$
where the last inequality comes from the fact that $1-r\leq \alpha$. Moreover, by the second result of Lemma \ref{integral}, we obtain that: $$\left(\frac{\alpha}{1-r}\right)^{\frac{1}{1-r}}\frac{1}{t\lambda_c(t)}\sim \frac{1}{t\gamma(t)}, \quad\text{ as $t\rightarrow+\infty$}.$$  This implies that for every $\varepsilon\in ]0,1[$ there exists $\hat{t}>t_0, \Lambda_{\varepsilon}\geq 1$ such that $\lambda_c(t)\leq \Lambda_{\varepsilon}\gamma(t)$ for every $t>\hat{t}$ $\left(\Lambda_{\varepsilon}=\left(\frac{\alpha}{1-r}\right)^{\frac{1}{1-r}}\frac{1}{(1-\varepsilon)}\right)$, this implies  \eqref{Hb}.
\end{itemize}
\end{proof}

\begin{remark}
  Finding all (or at least a larger class of) continuous functions $\gamma$ that satisfy \eqref{Ha} and for which one can prove \eqref{Hb} in general is an open problem.
\end{remark}

\subsection{Tighter convergence rates of the values}
In order to illustrate the context of the following result, it is useful to mention that if $\gamma(t)=\frac{\alpha}{t}$, then Theorem \ref{minimization} gives us minimization properties in the case $\alpha>0$. However, as mentioned in the introduction, it is widely known in the continuous deterministic setting \eqref{AVD}, with $\gamma(t)=\frac{\alpha}{t}$ and $\alpha>3$, then the values converge at the rate $o(1/t^2)$ (see \cite{cabot, faster1k2}). Based on \cite{cabot}, we will depict that effect for a general $\gamma$ in the continuous stochastic setting.

\smallskip

We will rephrase assumption \eqref{H0} on the objective $f$ to:
\begin{align}\label{H00}
\begin{cases}
\text{$f$ is convex and continuously differentiable with $L$-Lipschitz continuous gradient}; \\
\text{$f\in C^2(\H)$ or $\H$ is finite-dimensional;}\\
\calS \eqdef \argmin (f)\neq\emptyset. \tag{$\mathrm{H}_0^{\star}$} 
\end{cases}
\end{align} 
\eqref{H00} coincides with \eqref{H0} in the infinite-dimensional case, but is weaker than \eqref{H0} when $\H$ is finite-dimensional.

\begin{theorem}\label{fastconv}
Assume that $f,\sigma$ and $\gamma$ satisfy assumptions \eqref{H00}, \eqref{H} and \eqref{H1}-\eqref{H30}, respectively. Let $\nu\geq 2$, and consider the dynamic \eqref{NAG} with initial data $X_0,V_0\in\Lp^{\nu}(\Omega;\H)$. 
Then, there exists a unique solution $(X,V)\in S_{\H\times\H}^{\nu}[t_0]$ of \eqref{NAG}, for every $\nu\geq 2$.  Additionally, if $\Gamma\sigma_{\infty}\in\Lp^2([t_0,+\infty[)$, then:

 \begin{enumerate}[label=(\roman*)]
    \item $\int_{t_2}^{+\infty} \Gamma(t)(f(X(t))-\min f+\Vert V(t)\Vert^2)dt<+\infty$ a.s..
    \item \label{gamma-2} $f(X(t))-\min f+\Vert V(t)\Vert^2=o\left(\frac{1}{\Gamma^2(t)}\right)$ a.s..
    \item $\mathbb{E}(f(X(t))-\min f+\Vert V(t)\Vert^2)=\mathcal{O}\left(\frac{1}{\Gamma^2(t)}\right).$
     
 \end{enumerate}
Moreover, assume that $\Gamma\notin \Lp^1([t_0,+\infty[)$, and let $\theta(t)\eqdef \int_{t_0}^t \Gamma(s)ds$. If also $\theta\sigma_{\infty}^2\in \Lp^1([t_0,+\infty[)$, then:
\begin{enumerate}[label=(\roman*),resume]
    \item \label{gammaint} $f(X(t))-\min f+\Vert V(t)\Vert^2=o\left(\frac{1}{\theta(t)}\right)$ a.s..
    \item $\mathbb{E}(f(X(t))-\min f+\Vert V(t)\Vert^2)=\mathcal{O}\left(\min\Big\{\frac{\int_{t_0}^t \frac{ds}{\Gamma(s)}}{\theta(t)},\frac{1}{\Gamma^2(t)}\Big\}\right).$
\end{enumerate}
\end{theorem}

\begin{remark}
The claim \ref{gamma-2} is new even in the deterministic case (\ie $\sigma(\cdot,\cdot)=0_{\K;\H}$). According to the first three items of the previous theorem, the conclusions of Remark \ref{remarkimp} are also valid, this is that when $\gamma(t)=\frac{\alpha}{t}$ with $\alpha>3$ and $t\sigma_{\infty}(t)\in\Lp^2([t_0,+\infty[)$, the previous theorem ensures fast convergence of the values, $\ie, \mathcal{O}(t^{-2})$ in expectation and $o(t^{-2})$ in almost sure sense. Besides, by Corollary \ref{atr}, when $\gamma(t)=\frac{\alpha}{t^r}$ with $r\in ]0,1[,\alpha\geq 1-r, $ and $t^r\sigma_{\infty}(t)\in\Lp^2([t_0,+\infty[)$, the previous theorem ensures convergence of the objective at a rate $\mathcal{O}\left(t^{-2r}\right)$ in expectation and $o\left(t^{-2r}\right)$ in almost sure sense, moreover by \eqref{gammaint} the convergence rate is actually $o\left(t^{-(r+1)}\right)$ in almost sure sense, which is faster than $o\left(t^{-2r}\right)$. Regarding \eqref{gammaint}, this can be seen as the extension of \cite[Theorem 3.6]{cabot} to the stochastic setting.   
\end{remark}

\begin{proof}
\begin{enumerate}[label=(\roman*)]
    \item Let $m<\frac{3}{2}$ and $t_2$ defined in \eqref{H30}, let also $b\in ]2(m-1),1[$ and $x^{\star}\in\calS$. Based on \eqref{systemabcd} with $\beta\equiv 0$,  we introduce $\phi_1:(t,x,v)\mapsto \R$ defined by 
    \[
    \phi_1(t,x,v)=\Gamma^2(t)(f(x)-\min f)+\frac{\Vert b(x-x^{\star})+\Gamma(t)v\Vert^2}{2}+\frac{b(1-b)}{2}\Vert x-x^{\star}\Vert^2.
    \] 
    Since $f\in C^2(\H),$ we use It\^o's formula from $t_2$ to $t$ to get
    \begin{equation}\label{lyapbeta0}
    \begin{aligned}
    \phi_1(t,X(t),V(t))&=\phi_1(t_2,X(t_2),V(t_2))+\int_{t_2}^t\Gamma(s)[2\Gamma'(s)(f(X(s))-\min f)-b\langle \nabla f(X(s)), X(s)-x^{\star} \rangle]ds\\
    &+(b-1)\int_{t_2}^t\Gamma(s)\Vert V(s)\Vert^2ds+\int_{t_2}^t \Gamma^2(s)\tr[\Sigma(s,X(s))]ds
    \\
    &+\underbrace{\int_{t_2}^t\langle [\Gamma^2(s)V(s)+b\Gamma(s)(X(s)-x^{\star})]\sigma^{\star}(s,X(s)), dW(s)\rangle}_{M_t}.
    \end{aligned}  
\end{equation}
When $\H$ is finite-dimensional but $f$ is not $C^2(\H)$, we can use mollifiers as in \cite[Proposition~C.2]{mertikopoulos_staudigl_2018}, and get \eqref{lyapbeta0} as an inequality in this case.

\smallskip

Besides, we have that $$\int_{t_2}^{+\infty} \Gamma^2(s)\tr[\Sigma(s,X(s))]ds\leq \int_{t_2}^{+\infty} \Gamma^2(s)\sigma_{\infty}^2(s)ds<+\infty.$$
Besides $(M_t)_{t\geq t_2}$ is a continuous martingale. Moreover, by convexity of $f$, we have that 
\begin{align*}
&\int_{t_2}^t\Gamma(s)[2\Gamma'(s)(f(X(s))-\min f)-b\langle \nabla f(X(s)), X(s)-x^{\star} \rangle]ds\\
&\leq \int_{t_2}^t\Gamma(s)(2\Gamma'(s)-b)(f(X(s))-\min f)ds.
\end{align*}
Since $b-1<0$, and $$2\Gamma'(t)-b=2\gamma(t)\Gamma(t)-2-b\leq 2(m-1)-b<0, \quad \forall t> t_2. $$ By Theorem~\ref{impp}, \begin{equation}\label{gfv}
    \int_{t_2}^{+\infty}\Gamma(s)(f(X(s))-\min f+\Vert V(s)\Vert^2)ds<+\infty \quad a.s.,
\end{equation} and 
$$\lim_{t\rightarrow +\infty}\Gamma^2(t)(f(X(t))-\min f)+\frac{\Vert b(X(t)-x^{\star})+\Gamma(t)V(t)\Vert^2}{2}+\frac{b(1-b)}{2}\Vert X(t)-x^{\star}\Vert^2 \text{ exists a.s..}$$ 

\item On the other hand, let $\phi_2:(t,x,v)\mapsto \R$ defined by $\phi_2(t,x,v)=\Gamma^2(t)\left(f(x)-\min f +\frac{\Vert v\Vert^2}{2}\right)$. Recalling the discussion for $\phi_1$, we get that by It\^o's formula from $t_2$ to $t$, we have  \begin{equation}\label{lyap2beta0}
\begin{aligned}
      \phi_2(t,X(t),V(t))&=\phi_2(t_2,X(t_2),V(t_2))+\int_{t_2}^t 2\Gamma(s)\Gamma'(s) (f(X(s))-\min f) ds\\
      &-\int_{t_2}^t \Gamma(s)\Vert V(s)\Vert^2ds+\int_{t_2}^t \Gamma^2(s)\tr[\Sigma(s,X(s))]ds\\
    &+\underbrace{\int_{t_2}^t\Gamma^2(s)\langle V(s)\sigma^{\star}(s,X(s)),dW(s)\rangle}_{M_t}.
\end{aligned}
\end{equation}
And also, that
\begin{align*}
&\int_{t_2}^{+\infty} 2\Gamma(s)\Gamma'(s) (f(X(s))-\min f)+\Gamma^2(s)\tr[\Sigma(s,X(s))]ds \\
&\leq\int_{t_2}^{+\infty} \Gamma(s)(f(X(s))-\min f)+\Gamma^2(s)\sigma_{\infty}^2(s)ds<+\infty \quad a.s..
\end{align*}
Besides $(M_t)_{t\geq t_2}$ is a continuous martingale. By Theorem~\ref{impp}, we get again that $\int_{t_2}^{+\infty}\Gamma(s)\Vert V(s)\Vert^2 ds<+\infty$ a.s. and that \begin{equation}\label{ggfv}
    \lim_{t\rightarrow +\infty}\Gamma^2(t)\left(f(X(t))-\min f+\frac{\Vert V(t)\Vert^2}{2}\right) \text{ exists a.s. }
    \end{equation}
\smallskip

 Let us recall that $\frac{1}{\Gamma}\notin\Lp^1([t_0,+\infty[)$ by Lemma \ref{1gam}. Therefore, by \eqref{gfv} and \eqref{ggfv}, we can use Lemma \ref{lim0} to obtain that $$\lim_{t\rightarrow +\infty}\Gamma^2(t)\left(f(X(t))-\min f+\frac{\Vert V(t)\Vert^2}{2}\right)=0 \quad a.s..$$

\item Taking expectation on \eqref{lyapbeta0} and denoting $$K_1\eqdef\Gamma^2(t_2)\mathbb{E}((f(X(t_2))-\min f))+\frac{1}{2}\mathbb{E}(\Vert b(X(t_2)-x^{\star})+\Gamma(t_2)V(t_2)\Vert^2)+\frac{b(1-b)}{2}\Vert X(t_2)-x^{\star}\Vert^2,$$ $$K_{\Gamma}\eqdef\int_{t_2}^{+\infty}\Gamma^2(s)\sigma_{\infty}^2(s)ds,$$ we obtain directly that $$\mathbb{E}\left(\Gamma^2(t)(f(X(t))-\min f)+\frac{\Vert b(X(t)-x^{\star})+\Gamma(t)V(t)\Vert^2}{2}+\frac{b(1-b)}{2}\Vert X(t)-x^{\star}\Vert^2\right)\leq K_1+K_{\Gamma}.$$
From this, is direct that $\sup_{t\geq t_2}\mathbb{E}(\Vert X(t)-x^{\star}\Vert^2)<+\infty$, and this in turn imply
$$\mathbb{E}\left(f(X(t))-\min f+\frac{\Vert V(t)\Vert^2}{2}\right)=\mathcal{O}\left( \frac{1}{\Gamma^2(t)}\right).$$
\item Moreover, assume that $\Gamma\notin \Lp^1([t_0,+\infty[)$, and let $\theta(t)=\int_{t_0}^t \Gamma(s)ds$. If also $\theta\sigma_{\infty}^2\in \Lp^1([t_0,+\infty[)$, then we consider $\phi_3(t,x,v)=\theta(t)\left(f(x)-\min f+\frac{\Vert v\Vert^2}{2}\right)$, by It\^o's formula from $t_2$ to $t$, we get \begin{equation}\label{itointegral}
    \begin{aligned}
        \phi_3(t,X(t),V(t))&=\phi_3(t_2,X(t_2),V(t_2))+\int_{t_2}^t \Gamma(s)\left(f(X(s))-\min f+\frac{\Vert V(s)\Vert^2}{2}\right)ds\\
        &-\int_{t_2}^t \gamma(s)\theta(s)\Vert V(s)\Vert^2ds+\frac{1}{2}\int_{t_2}^t \theta(s)\tr[\Sigma(s,X(s))]ds\\
        & +\underbrace{\int_{t_2}^t\theta(s)\langle V(s)\sigma^{\star}(s,X(s)),dW(s)\rangle}_{M_t}.
    \end{aligned}
\end{equation}
Also, by the first item and new hypothesis on the diffusion term, we get that \begin{equation}
    \begin{aligned}
        &\int_{t_2}^t \Gamma(s)\left(f(X(s))-\min f+\frac{\Vert V(s)\Vert^2}{2}\right)+\theta(s)\tr[\Sigma(s,X(s))]ds\\
        &\leq \int_{t_2}^{+\infty} \Gamma(s)\left(f(X(s))-\min f+\frac{\Vert V(s)\Vert^2}{2}\right)+\theta(s)\sigma_{\infty}^2(s)ds<+\infty.
    \end{aligned}
\end{equation}
Besides $(M_t)_{t\geq t_2}$ is a continuous martingale. By Theorem~\ref{impp}, we get that $\int_{t_2}^{+\infty}\gamma(s)\theta(s)\Vert V(s)\Vert^2 ds<+\infty$ a.s. and that \begin{equation}
    \lim_{t\rightarrow +\infty}\theta(t)\left(f(X(t))-\min f+\frac{\Vert V(t)\Vert^2}{2}\right) \text{ exists a.s., }
    \end{equation}
    Using Lemma \ref{intgam} with $q(t)=\theta(t)$, we get that $\frac{\Gamma}{\theta}\notin \Lp^1([t_2,+\infty[)$. Besides, recalling that $$\int_{t_2}^{+\infty} \Gamma(s)\left(f(X(s))-\min f+\frac{\Vert V(s)\Vert^2}{2}\right)<+\infty, \quad a.s.,$$
    we invoke Lemma \ref{lim0} to conclude that $\lim_{t\rightarrow +\infty} \theta(t)\left(f(X(t))-\min f+\frac{\Vert V(t)\Vert^2}{2}\right)=0$ a.s..
    \item Taking expectation in \eqref{itointegral} and upper bounding we get \begin{equation}
    \begin{aligned}
        \mathbb{E}(\phi_3(t,X(t),V(t)))&\leq \mathbb{E}(\phi_3(t_2,X(t_2),V(t_2)))+\int_{t_2}^t \Gamma(s)\mathbb{E}\left(f(X(s))-\min f+\frac{\Vert V(s)\Vert^2}{2}\right)ds\\
        &+\frac{1}{2}\int_{t_2}^{+\infty} \theta(s)\sigma_{\infty}^2(s)ds.
    \end{aligned}
\end{equation}
By the third item, we have that $\mathbb{E}\left(f(X(s))-\min f+\frac{\Vert V(s)\Vert^2}{2}\right)=\mathcal{O}\left(\frac{1}{\Gamma^2(s)}\right)$, so we conclude that \begin{equation}
        \mathbb{E}\left(\theta(t)\left(f(X(t))-\min f+\frac{\Vert V(t)\Vert^2}{2}\right)\right)=\mathcal{O}\left( \int_{t_2}^t \frac{ds}{\Gamma(s)}\right). 
\end{equation}
Thus, $$\mathbb{E}\left(f(X(t))-\min f+\frac{\Vert V(t)\Vert^2}{2}\right)=\mathcal{O}\left(\min\Big\{\frac{\int_{t_2}^t \frac{ds}{\Gamma(s)}}{\theta(t)},\frac{1}{\Gamma^2(t)}\Big\}\right).$$
\end{enumerate}
\end{proof}
\subsection{Almost sure weak convergence of trajectories}
In the deterministic setting with $\alpha>3$, it is also well-known that one can obtain weak convergence of the trajectory. Our aim in this section is to show this claim for a general $\gamma$ in the stochastic setting.
\begin{theorem}\label{almostsureweak}
    Consider the setting of Theorem \ref{fastconv}. Then, if $\Gamma\sigma_{\infty}\in \Lp^2([t_0,+\infty[)$ we have that:
    \begin{enumerate}[label=(\roman*)]
        \item $ \mathbb{E}\left[\sup_{t\geq t_2}\Vert X(t)\Vert^{\nu}\right]<+\infty$.
        \item \label{iiconv0} $\forall x^{\star}\in\calS$, $\lim_{t\rightarrow +\infty} \Vert X(t)-x^{\star}\Vert$ exists a.s.. 
      \vspace{1mm}
     \item  
     If $\gamma$ is non-increasing, there exists an $\mathcal{S}-$valued random variable $X^{\star}$ such that $\wlim_{t\rightarrow +\infty} X(t) = X^{\star}$ a.s..
    \end{enumerate}
\end{theorem}
\begin{proof}
    \begin{enumerate}[label=(\roman*)]
    \item Analogous to the proof of the first point of \cite[Theorem 3.1]{mio}.
    \item Recalling the proof of Theorem \ref{fastconv}, we combine the fact that both 
    \[
    \lim_{t\rightarrow +\infty}\Gamma^2(t)(f(X(t))-\min f)+\frac{\Vert b(X(t)-x^{\star})+\Gamma(t)V(t)\Vert^2}{2}+\frac{b(1-b)}{2}\Vert X(t)-x^{\star}\Vert^2,
    \]
and
\[
\lim_{t\rightarrow +\infty}\Gamma^2(t)\left(f(X(t))-\min f+\frac{\Vert V(t)\Vert^2}{2}\right)
\] 
exist a.s.. We can substract both quantities to obtain that 
\[
\lim_{t\rightarrow +\infty} \frac{\Vert X(t)-x^{\star}\Vert^2}{2}+\Gamma(t)\langle X(t)-x^{\star}, V(t) \rangle \text{ exists a.s..}
\]
Thus, for every $x^{\star}\in\calS$ there exists $\Omega_{x^{\star}}\in\mathcal{F}$ with $\Pro(\Omega_{x^{\star}})=1$ and $\exists \ell:\Omega_{x^{\star}}\mapsto\R$ such that 
\[
\lim_{t\rightarrow +\infty}\frac{\Vert X(\omega,t)-x^{\star}\Vert^2}{2}+\Gamma(t)\langle V(\omega,t)  , X(\omega,t)-x^{\star}\rangle=\ell(\omega).
\]
Let $Z(\omega,t)=\frac{\Vert X(\omega,t)-x^{\star}\Vert^2}{2}-\ell(\omega)$ and $\varepsilon>0$ arbitrary. There exists $T(\omega)\geq t_0$ such that $\forall t\geq T(\omega)$ $$\Big\Vert Z(\omega,t)+\Gamma(t)\langle V(\omega,t),X(\omega,t)-x^{\star}\rangle\Big\Vert<\varepsilon.$$
    
    Let $g(t)\eqdef \exp\left(\int_{t_2}^t \frac{ds}{\Gamma(s)}\right)$, multiplying the previous inequality by $\frac{g(t)}{\Gamma(t)}$, there exists $T(\omega)\geq t_0$ such that for every $t\geq T(\omega)$: $$\Big\Vert \frac{g(t)}{\Gamma(t)}Z(t)+g(t)\langle V(\omega,t),X(\omega,t)-x^{\star}\rangle\Big\Vert<\frac{\varepsilon}{\Gamma(t)}g(t).$$
    
    On the other hand, $dZ(t)=\langle V(t), X(t)-x^{\star}\rangle dt$ and $$d\left(g(t)Z(t)\right)=\left(\frac{g(t)}{\Gamma(t)}Z(t)+g(t)\langle V(t),X(t)-x^{\star}\rangle\right)dt.$$
    Thus, \begin{align*}
        \Vert g(t)Z(t)-g(T)\tcb{Z}(T)\Vert&=\Big\Vert\int_T^t d(g(s)Z(s))\Big\Vert=\Big\Vert\int_T^t \left(\frac{g(s)}{\Gamma(s)}Z(s)+g(s)\langle V(s),X(s)-x^{\star}\rangle\right)ds\Big\Vert \\
        &\leq \varepsilon\int_T^t \frac{g(s)}{\Gamma(s)}ds=\varepsilon(g(t)-g(T)). 
    \end{align*}
    So, $$\Vert Z(t)\Vert\leq \frac{g(T)}{g(t)}\Vert \tcb{Z}(T)\Vert+\varepsilon.$$
    By Lemma \ref{1gam}, we obtain that $\lim_{t\rightarrow +\infty}g(t)=+\infty$. Hence, $\limsup_{t\rightarrow +\infty} \Vert Z(t)\Vert\leq \varepsilon$. And we conclude that for every $x^{\star}\in\calS$, $\lim_{t\rightarrow +\infty} \frac{\Vert X(t)-x^{\star}\Vert}{2}$ exists a.s.. By a separability argument (see proof of \cite[Theorem 3.1]{mio} or \cite[Theorem 3.6]{mio2}) there exists $\tilde{\Omega}\in\mathcal{F}$ (independent of $x^{\star}$) such that $\Pro(\tilde{\Omega})=1$ and $\lim_{t\rightarrow +\infty} \frac{\Vert X(\omega,t)-x^{\star}\Vert}{2}$ exists for every $\omega\in \tilde{\Omega}, x^{\star}\in\calS$.
    
    \smallskip
    
    \item If $\gamma$ is non-increasing, then $\Gamma$ is non-decreasing (see \cite[Corollary 2.3]{cabot}). Then, by item \ref{gamma-2} of Theorem \ref{fastconv}, we have that: $$\lim_{t\rightarrow+\infty}f(X(t))=\min f \quad a.s..$$  Let $\Omega_f\in\mathcal{F}$ be the set of events on which this limit is satisfied. Thus $\Pro(\Omega_f)=1$. Set $\Omega_{\mathrm{conv}}\eqdef\Omega_f\cap\tilde{\Omega}$. We have $\Pro(\Omega_{\mathrm{conv}})=1$. Now, let $\omega\in\Omega_{\mathrm{conv}}$ and $\widetilde{X}(\omega)$ be a weak sequential cluster point of $X(\omega,t)$ (\tcb{which exists by boundedness on $\Omega_{\mathrm{conv}}$}). Equivalently, there exists an increasing sequence $ (t_k)_{k\in\N}\subset \R_+$ such that $\lim_{k\rightarrow +\infty} t_k=+\infty$, and 
\[
\wlim_{k\rightarrow +\infty} X(\omega, t_k) = \widetilde{X}(\omega).
\]
Since $\lim_{t\rightarrow +\infty} f(X(\omega,t))=\min f$ and the fact that $f$ is weakly lower semicontinuous (since it is convex and continuous), we obtain directly that $\widetilde{X}(\omega)\in \calS$. Finally by Opial's Lemma (see \cite{opial}) we conclude that there exists $X^{\star}(\omega)\in \calS$ such that $\wlim_{t\rightarrow +\infty}X(\omega,t)= X^{\star}(\omega)$. In other words, since $\omega\in\Omega_{\mathrm{conv}}$ was arbitrary, there exists an $\calS$-valued random variable $X^{\star}$ such that $\wlim_{t\rightarrow +\infty} X(t)= X^{\star}$ a.s..
    \end{enumerate}
\end{proof}

\appendix

\section{Auxiliary results}\label{aux}
\subsection{Deterministic results}
\begin{lemma}\label{axby}
    Let $a,b\in\R$ and $x,y\in\H$, then $$\Vert ax-by\Vert\leq \max\{|a|,|b|\}\Vert x-y\Vert+|a-b|\max\{\Vert x\Vert,\Vert y\Vert\}.$$
\end{lemma}
\begin{lemma}\label{lim0}
Let $t_0>0$ and $a,b:[t_0,+\infty[\rightarrow \R_+$. If $\lim_{t\rightarrow \infty} a(t)$ exists, $b\notin\Lp^1([t_0,+\infty[)$ and $\int_{t_0}^{+\infty} a(s)b(s)ds<+\infty$, then $\lim_{t\rightarrow \infty} a(t)=0.$
\end{lemma}
\begin{lemma}\label{1gam}
Under hypothesis \eqref{H1}, then $$\int_{t_0}^{+\infty} \frac{ds}{\Gamma(s)}=+\infty.$$
\end{lemma}

\begin{proof}
Let $q(t)\eqdef \int_t^{+\infty}\frac{ds}{p(s)}$, since $\int_{t_0}^{+\infty}\frac{ds}{p(s)}<+\infty$, then $\lim_{t\rightarrow +\infty} q(t)=0$ and $q'(t)=-\frac{1}{p(t)}$. On the other hand \begin{align*}
    \int_{t_0}^{+\infty} \frac{ds}{\Gamma(s)}=-\int_{t_0}^{+\infty}\frac{q'(s)}{q(s)}ds=\ln(q(t_0))-\lim_{t\rightarrow +\infty}\ln(q(t))=+\infty.
\end{align*}
\end{proof}
\begin{lemma}\label{intgam}
    Let $q:[t_0,+\infty[\rightarrow\R_+$ be a non-decreasing differentiable function, if $q\notin\Lp^1([t_0,+\infty[)$, then $\frac{q'}{q}\notin\Lp^1([t_0,+\infty[)$
\end{lemma}
\begin{proof}
    By definition, $$\int_{t_0}^{+\infty} \frac{q'(s)}{q(s)}ds=\lim_{t\rightarrow +\infty}\ln(q(t))-\ln(q(t_0))=+\infty.$$
\end{proof}

\begin{lemma}\label{incomplete}
    For $a,x> 0$, let us define the upper incomplete Gamma function as:
    $$\Gamma_{inc}(a;x)=\int_x^{+\infty} s^{a-1}e^{-s}ds.$$
    Then, the following holds:
    \begin{enumerate}[label=(\roman*)]
        \item $x^{1-a}e^x\Gamma_{inc}(a;x)\leq 1$ for $0<a\leq 1$.
        \item \label{ine} $x^{1-a}e^x\Gamma_{inc}(a;x)\geq 1$ for $a\geq 1$.
        \item \label{asy} $\lim_{x\rightarrow +\infty} x^{1-a}e^x\Gamma_{inc}(a;x)=1$
    \end{enumerate}
\end{lemma}
\begin{proof}
    See \cite[Section 8]{dawson}.
\end{proof}
\begin{remark}
    Do not confuse $\Gamma_{inc}(a;x)$ with $\Gamma(t)$ defined in \eqref{defgamma}.
\end{remark}
\begin{corollary}\label{atr}
    Let us consider the viscous damping function $\gamma:[t_0,+\infty[\rightarrow\R_+$ defined by $\gamma(t)=\frac{\alpha}{t^r}$ with $r\in ]0,1[$ and $\alpha\geq 1-r$, then:
    \begin{enumerate}[label=(\roman*)]
        \item $\gamma$ satisfies \eqref{H1}.
        \item $\Gamma(t)=\mathcal{O}(t^r)$.
        \item $\gamma$ satisfies \eqref{H30}.
    \end{enumerate}
\end{corollary}
\begin{proof}
\begin{enumerate}[label=(\roman*)]
    \item Let $c\eqdef \frac{\alpha}{1-r}\geq 1$, we first notice that after the change of variable $u=cs^{1-r}$, we get $$\int_{0}^{+\infty} \exp(-cs^{1-r})ds=\frac{1}{\alpha c^{\frac{r}{1-r}}}\int_0^{+\infty}u^{\frac{1}{1-r}-1}e^{-u}du<+\infty,$$
    since the last integral is the classical Gamma function (see \eg \cite[Section 5]{dawson}) evaluated at $\frac{1}{1-r}$, and this function is well defined for positive arguments, then \eqref{H1} is satisfied \tcb{as $t_0>0$}. 
    \item Besides, by definition $\Gamma(t)=\exp(ct^{1-r})\int_t^{+\infty}\exp(-cs^{1-r})ds$.
    Using the same change of variable as before, we obtain that \begin{equation}
        \Gamma(t)=\frac{\exp(ct^{1-r})}{\alpha c^{\frac{r}{1-r}}}\Gamma_{inc}\left(\frac{1}{1-r},ct^{1-r}\right).
    \end{equation}
    By \ref{asy} of Lemma \ref{incomplete} with $a=\frac{1}{1-r}>1$ and $x=ct^{1-r}$, for every $\varepsilon>0$, there exists $t_1>t_0$ such that for every $t>t_1$: $$\Gamma(t)\leq \frac{1+\varepsilon}{\alpha c^{\frac{1}{1-r}}}t^r.$$
\item Moreover, if we restrict $\varepsilon\in ]0,\frac{1}{2}[$, there exists $t_1>t_0$ such that for every $t>t_1$: $$\gamma(t)\Gamma(t)\leq \frac{1+\varepsilon}{c^{\frac{1}{1-r}}}\leq 1+\varepsilon.$$
Defining $m$ as $1+\varepsilon$, we have that $m<\frac{3}{2}$, and we conclude.
    \end{enumerate}
\end{proof}
\begin{lemma}\label{integral}
Let us define $p>0$ and $I_p(t)\eqdef \int_0^1 e^{-tu}(1-u)^p du$. Then 
\begin{enumerate}[label=(\roman*)]
\item $I_p(t)\leq t^{-1}$ for every $t>0$. 
\item $I_p(t)\sim t^{-1}$ as $t\rightarrow+\infty$.
\end{enumerate}
\end{lemma}
\begin{proof}
The first result comes from bounding the term $(1-u)^p$ by $1$ in the integral, then we can notice directly that $I_p(t)\leq t^{-1}$ for every $t>0$. The second result is an application of Watson's Lemma (see \cite{watson}).
\end{proof}

\subsection{Stochastic results}
\subsubsection{On stochastic processes}\label{onstochastic}
We refer to the notation and results discussed in \cite[Section A.2]{mio2}.

\smallskip
\begin{proposition}
    \label{existencecor}
    Consider $\nu\geq 2$, $X_0,V_0\in \Lp^{\nu}(\Omega;\H)$, $f$ and $\sigma$ satisfying \eqref{H0} and \eqref{H}, respectively. Consider also $\gamma$ satisfying \eqref{H1}, and $\beta$ satisfying \eqref{H2}. Then \eqref{ISIHD-S} has a unique solution $(X,V)\in S_{\H\times\H}^{\nu}[t_0]$.
\end{proposition}
\begin{remark}
    Hypothesis \eqref{H2} does not allow us to consider the case $\beta\equiv 0$, nevertheless, this case is well studied in Section \ref{sec:nesterov0}.
\end{remark}
\begin{proof}
    We rewrite \eqref{ISIHD-S} as in the reformulation \eqref{refor}, we recall \eqref{refor2}
 \begin{equation*}\begin{cases}
    dZ(t)&=-\beta(t)\nabla \mathcal{G}(Z(t))dt-\mathcal{D}(t,Z(t))dt+\hat{\sigma}(t,Z(t))dW(t),\quad t>t_0,\\
    Z(t_0)&=(X_0,X_0+\beta(t_0)V_0),\end{cases}
\end{equation*}
Since $\beta(t)\leq c_1$, we have that $-\beta(t)\nabla \mathcal{G}(Z(t))$ is Lipschitz, besides, since $\Big|\frac{\beta'(t)-\gamma(t)\beta(t)+1}{\beta(t)}\Big|\leq c_2$, we have that $\mathcal{D}$ is a Lipschitz operator. Then, using the hypotheses on $\sigma$, we can use \cite[Theorem 3.3]{mio2} and conclude the existence and uniqueness of a process $Z\in S_{\H\times \H}^{\nu}[t_0]$, this, in turn, implies the existence and uniqueness of a solution $(X,V)\in  S_{\H\times \H}^{\nu}[t_0]$ of \eqref{ISIHD-S}.  
\end{proof}
\begin{theorem}\label{convmartingale}
 Let $\H$ be a separable Hilbert space and $(M_t)_{t\geq 0}:\Omega\rightarrow\H$ be a continuous martingale such that $\sup_{t\geq 0} \EE\pa{\Vert M_t\Vert^2}<+\infty$. Then there exists a $\H-$valued random variable $M_{\infty}\in\Lp^2(\Omega;\H)$ such that $\lim_{t\rightarrow \infty}  M_t =M_{\infty}$ a.s..    
\end{theorem}
\begin{proof}
    For the proof, we refer to \cite[Theorem A.7]{mio2}.
\end{proof}

\begin{theorem} \label{impp} \cite[Theorem 1.3.9]{mao}
 Let $\{A_t\}_{t\geq 0} $ and $\{U_t\}_{t\geq 0} $ be two continuous adapted increasing processes with $A_0=U_0=0$ a.s.. Let $\{M_t\}_{t\geq 0} $ be a real-valued continuous local martingale with $M_0=0$ a.s.. Let $\xi$ be a non-negative $\mathcal{F}_0$-measurable random variable. Define  $$X_t=\xi+A_t-U_t+M_t\hspace{0.3cm} \text{ for } t\geq 0.$$
 If $X_t$ is non-negative and $\lim_{t\rightarrow +\infty} A_t<\infty$, then   $\lim_{t\rightarrow +\infty} X_t$ exists and is finite, and $\lim_{t\rightarrow +\infty} U_t<\infty$.
\end{theorem}
\subsection{Abstract integral bounds, almost sure and in expectation properties of \texorpdfstring{\eqref{ISIHD-S}}{}}
In the following proposition, we state different abstract integral bounds and almost sure properties for \eqref{ISIHD-S}, finally concluding with the almost sure convergence of the gradient towards zero. 
\begin{proposition}\label{abcdcomplete}
Consider that $f,\sigma$ satisfy \eqref{H0} and \eqref{H}, respectively. Let $\nu\geq 2$, and consider the dynamic \eqref{ISIHD-S} with initial data $X_0,V_0\in\Lp^{\nu}(\Omega;\H)$. Consider also $\gamma, \beta$ from \eqref{ISIHD-S} satisfying \eqref{H1} and \eqref{H2}, respectively, and suppose there exists $a,b,c,d$ satisfying \eqref{systemabcd}. Finally, we consider $\mathcal{E}$ the energy function defined in \eqref{liapunov}.

\smallskip

Then, there exists a unique solution $(X,V)\in S_{\H\times \H}^{\nu}[t_0]$ of \eqref{ISIHD-S}. Moreover, if   $t\mapsto m(t)\sigma_{\infty}^2(t)\in\Lp^1([t_0,+\infty[)$, where $m(t)\eqdef\max\{1,a(t),c^2(t)\}$, then the following properties are satisfied: \begin{enumerate}[label=(\roman*)]
    \item \label{one}$$\lim_{t\rightarrow+\infty}\mathcal{E}(t,X(t),V(t)) \hspace{0.2cm} \text{exists a.s..}$$
    \item $\int_{t_0}^{+\infty}(b(s)c(s)-a'(s))(f(X(s)+\beta(s)V(s))-\min f)ds<+\infty$, a.s..
    \item $\int_{t_0}^{+\infty}a(s)\beta(s)\Vert\nabla f(X(s))+\beta(s)V(s)\Vert^2ds<+\infty$, a.s..
    \item $\int_{t_0}^{+\infty}\left(b(s)b'(s)+\frac{d'(s)}{2}\right)\Vert X(s)-x^{\star}\Vert^2ds<+\infty$, a.s..
    \item \label{five} $\int_{t_0}^{+\infty}c(s)(\gamma(s)c(s)-c'(s)-b(s))\Vert V(s)\Vert^2ds<+\infty$, a.s..
    \item \label{six} 
    If $b(t)c(t)-a'(t)=\mathcal{O}(c(t)(\gamma(t)c(t)-c'(t)-b(t)))$, then $$\int_{t_0}^{+\infty}(b(s)c(s)-a'(s))(f(X(s))-\min f)ds<+\infty \quad a.s..$$ 
\item \label{v0} If there exists $\eta>0,\hat{t}>t_0$ such that $$\eta\leq c(t)(\gamma(t)c(t)-c'(t)-b(t)),\quad \eta\leq a(t)\beta(t),\quad \gamma(t)\leq \eta,\quad \forall t>\hat{t},$$ then $\lim_{t\rightarrow +\infty}\Vert V(t)\Vert=0$ a.s., $\lim_{t\rightarrow +\infty}\Vert \nabla f(X(t)+\beta(t)V(t))\Vert=0$ a.s., and 
$\lim_{t\rightarrow +\infty}\Vert \nabla f(X(t))\Vert=0$ a.s.

\end{enumerate}
\end{proposition}
\begin{proof}
The existence and uniqueness of a solution is a direct consequence of Corollary \ref{existencecor}. Moreover, applying Proposition \ref{itos} with $\mathcal{E}$, we can obtain \begin{equation}\label{eqhessianpert}
    \begin{split}
        \mathcal{E}(t,X(t),V(t))&\leq\mathcal{E}(t_0,X_0,V_0)-\int_{t_0}^t (b(s)c(s)-a'(s))(f(X(s)+\beta(s)V(s))-\min f)ds\\
        &-\int_{t_0}^t a(s)\beta(s)\Vert \nabla f(X(s)+\beta(s)V(s))\Vert^2ds-\int_{t_0}^t \left(b(s)b'(s)+\frac{d'(s)}{2}\right)\Vert X(s)-x^{\star}\Vert^2ds\\
        &-\int_{t_0}^t c(s)(b(s)+c'(s)-c(s)\gamma(s))\Vert V(s)\Vert^2ds+\int_{t_0}^t (La(s)\beta^2(s)+c^2(s))\sigma_{\infty}^2(s)ds+M_t,
    \end{split}
\end{equation}
where $M_t=\int_{t_0}^t\langle \sigma^{\star}(s,X(s)+\beta(s)V(s))(a(s)\beta(s)\nabla f(X(s)+\beta(s)V(s))+c(s)[b(s)(X(s)-x^{\star})+c(s)V(s)], dW(s)\rangle$.
Since $\sup_{t\in [t_0,T]}\mathbb{E}(\Vert X(t)\Vert^2)<+\infty, \sup_{t\in [t_0,T]}\mathbb{E}(\Vert V(t)\Vert^2)<+\infty$ for every $T>t_0$, and $a,b,c,\beta$ are continuous functions, we have that $M_t$ is a continuous martingale, on the other hand, we have that $$\int_{t_0}^{+\infty} (La(s)\beta^2(s)+c^2(s))\sigma_{\infty}^2(s)ds<+\infty.$$ Then, we can apply Theorem \ref{impp} and conclude that $\lim_{t\rightarrow +\infty}\mathcal{E}(t,X(t),V(t))$ exists a.s. and \begin{itemize}
    \item $\int_{t_0}^{+\infty}(b(s)c(s)-a'(s))(f(X(s)+\beta(s)V(s))-\min f)ds<+\infty$ a.s..
    \item $\int_{t_0}^{+\infty}a(s)\beta(s)\Vert\nabla f(X(s))+\beta(s)V(s)\Vert^2ds<+\infty$ a.s..
    \item $\int_{t_0}^{+\infty}\left(b(s)b'(s)+\frac{d'(s)}{2}\right)\Vert X(s)-x^{\star}\Vert^2ds<+\infty$ a.s..
    \item $\int_{t_0}^{+\infty}c(s)(\gamma(s)c(s)-c'(s)-b(s))\Vert V(s)\Vert^2ds<+\infty$ a.s..
\end{itemize}
This let us conclude with items \ref{one} to \ref{five}.

\smallskip

Let $\tilde{b}(t)=b(t)c(t)-a'(t)$, and
    \begin{align*}
        I_f&\eqdef\int_{t_0}^{+\infty} \tilde{b}(s)(f(X(s))-\min f) ds\\
        &\leq\int_{t_0}^{+\infty} \tilde{b}(s)(f(X(s))-f(X(s)+\beta(s)V(s))) ds+\int_{t_0}^{+\infty} \tilde{b}(s)(f(X(s)+\beta(s)V(s))-\min f) ds\\ 
        \end{align*}
        Using Descent Lemma, Cauchy Schwarz Inequality and Corollary \ref{2L}:
        \begin{align*}
        I_f&\leq \sqrt{2L}\beta_0\pa{\int_{t_0}^{+\infty}\tilde{b}(s)(f(X(s)+\beta(s)V(s))-\min f)ds}^{\frac{1}{2}}\pa{\int_{t_0}^{+\infty}\tilde{b}(s)\Vert V(s)\Vert^2ds}^{\frac{1}{2}}\\
        &+\frac{L\beta_0^2}{2}\int_{t_0}^{+\infty}\tilde{b}(s)\Vert V(s)\Vert^2ds+\int_{t_0}^{+\infty} \tilde{b}(s)(f(X(s)+\beta(s)V(s))-\min f) ds.
    \end{align*}
  \noindent  If $\tilde{b}(t)=b(t)c(t)-a'(t)=\mathcal{O}(c(t)(\gamma(t)c(t)-c'(t)-b(t)))$, we have that $$\int_{t_0}^{+\infty} \tilde{b}(s)\Vert V(s)\Vert^2ds<+\infty \quad a.s..$$ And we conclude with item \ref{six}. 

  \smallskip

To prove \ref{v0}, in particular that $\lim_{t\rightarrow \infty}\Vert V(t)\Vert=0$,  we consider that if there exists $\eta>0,\hat{t}>t_0$ such that $\eta\leq c(t)(\gamma(t)c(t)-c'(t)-b(t)), \forall t>\hat{t}$, then there exists $\Omega_v\in\calF$ such that $\Pro(\Omega_v)=1$ and $$\int_{t_0}^{+\infty} \Vert V(\omega,s)\Vert^2ds<+\infty, \quad \forall \omega\in\Omega_v.$$ Then, we have $\liminf_{t\rightarrow \infty}\Vert V(\omega,t)\Vert=0, \forall \omega\in\Omega_v$. Let us suppose that $\limsup_{t\rightarrow \infty}\Vert V(\omega,t)\Vert>0, \forall\omega\in\Omega_v$. Then, by \cite[Lemma A.3]{mio2}, there exists $\delta>0$ satisfying 
\[
0=\liminf_{t\rightarrow +\infty} \norm{V(\omega,t)}<\delta<\limsup_{t\rightarrow +\infty} \norm{V(\omega,t)} ,\quad\forall\omega\in\Omega_v.
\] 
And there exists $(t_k)_{k\in\N}\subset [t_0,+\infty[$ such that $\lim_{k\rightarrow +\infty} t_k=+\infty$, 
\[ 
\norm{V(\omega,t_k)}>\delta,\quad \forall\omega\in\Omega_v \qandq t_{k+1}-t_k>1, \quad \forall k\in\N.
\]
Let $N_t\eqdef\int_{t_0}^t \sigma(s,X(s)+\beta(s)V(s))dW(s)$. This is a continuous martingale (w.r.t. the filtration $\calF_t$), which verifies
\[
\EE(\Vert N_t\Vert^2)=\EE\pa{\int_{t_0}^t \norm{\sigma(s,X(s)+\beta(s)V(s))}_{\HS}^2ds}\leq \EE\pa{\int_{t_0}^{+\infty} \sigma_{\infty}^2(s)ds}<+\infty, \forall t\geq t_0.
\]
 According to Theorem \ref{convmartingale}, we deduce that there exists a $\H-$valued random variable $N_{\infty}$ w.r.t. $\calF_{\infty}$, and which verifies: $\EE(\Vert N_{\infty}\Vert^2)<+\infty$, and there exists $\Omega_N\in\calF$ such that $\Pro(\Omega_N)=1$ and 
\[
\lim_{t\rightarrow +\infty}N_t(\omega)= N_{\infty}(\omega) \mbox{ for every } \omega\in\Omega_N .
\]
Let $\omega_0\in\Omega_{nv}\eqdef\Omega_N\cap\Omega_v$ ($\Pro(\Omega_{nv})=1$) and the notation $V(t)\eqdef V(\omega_0,t)$, $\varepsilon\in \left(0,\min\{1,\frac{\delta^2}{4}\}\right)$ arbitrary and recall that $\eta\leq a(t)\beta(t),\gamma(t)\leq \eta$ for every $t>\hat{t}$. Let $k'\in\N$ be such that $t_{k'}>\hat{t}$, $k>k'$ and $t\in[t_k,t_k+\varepsilon]$, then  \begin{align*}
    \norm{V(t)-V(t_k)}^2&\leq 3(t-t_k)\int_{t_k}^t\gamma^2(s)\Vert V(s)\Vert^2ds+3(t-t_k)\int_{t_k}^t \Vert\nabla f(X(s)+\beta(s)V(s))\Vert^2ds\\
    &\phantom{=}+3\Vert N_t-N_{t_k}\Vert^2\\
    &\leq 3\eta^2(t-t_k)\int_{t_k}^t\Vert V(s)\Vert^2ds+\frac{3}{\eta}(t-t_k)\int_{t_k}^t a(s)\beta(s)\Vert\nabla f(X(s)+\beta(s)V(s))\Vert^2ds\\
    &\phantom{=}+3\Vert N_t-N_{t_k}\Vert^2\\
    &\leq  3\eta^2\varepsilon\int_{t_k}^t\Vert V(s)\Vert^2ds+\frac{3}{\eta}\varepsilon\int_{t_k}^t a(s)\beta(s)\Vert\nabla f(X(s)+\beta(s)V(s))\Vert^2ds\\
    &\phantom{=}+3\Vert N_t-N_{t_k}\Vert^2.
\end{align*}
Now let $k''\in\N$ be such that for every $k>k''$, $$\int_{t_{k}}^{+\infty}\Vert V(s)\Vert^2ds<\frac{1}{9\eta^2}, \int_{t_{k}}^{+\infty} a(s)\beta(s)\Vert\nabla f(X(s)+\beta(s)V(s))\Vert^2ds< \frac{\eta}{9}, \sup_{t>t_k}\Vert N_t-N_{t_k}\Vert^2<\frac{\varepsilon}{9}.$$

Then, we have that $$\norm{V(t)-V(t_k)}^2\leq \varepsilon\leq\frac{\delta^2}{4},\forall t\in[t_k,t_k+\varepsilon], k>\max\{k',k''\}.$$ For such $t$, we bound using the triangular inequality and obtain $$\Vert V(t)\Vert\geq \Vert V(t_k)\Vert-\Vert V(t)-V(t_k)\Vert>\frac{\delta}{2}.$$ 
Now we consider $$\int_{t_0}^{+\infty}\Vert V(s)\Vert^2ds\geq \sum_{k>\max\{k',k''\}}\int_{t_k}^{t_k+\varepsilon}\Vert V(s)\Vert^2 ds\geq \sum_{k>\max\{k',k''\}} \frac{\varepsilon\delta^2}{4}=+\infty.$$
Which is a contradiction, then we conclude that $\liminf_{t\rightarrow +\infty}\Vert V(t)\Vert=\limsup_{t\rightarrow +\infty}\Vert V(t)\Vert=0$, a.s..\smallskip

To prove the second part of \ref{v0}, i.e. that $\lim_{t\rightarrow +\infty}\Vert\nabla f(X(t)+\beta(t)V(t))\Vert=0$, we recall that there exists $\eta>0,\hat{t}>t_0$ such that $\eta\leq a(t)\beta(t)$, then there exists $\Omega_y\in\calF$ such that $\Pro(\Omega_y)=1$ and $$\int_{t_0}^{+\infty} \Vert\nabla f(X(\omega,s)+\beta(s)V(\omega,s))\Vert^2ds<+\infty \quad \forall\omega\in\Omega_y$$
So we have that $$\liminf_{t\rightarrow +\infty} \Vert\nabla f(X(\omega,t)+\beta(t)V(\omega,t))\Vert=0, \quad \forall\omega\in\Omega_y.$$ Moreover, if we suppose that $$\limsup_{t\rightarrow +\infty} \Vert\nabla f(X(\omega,t)+\beta(t)V(\omega,t))\Vert>0,\forall\omega\in\Omega_y,$$  by \cite[Lemma A.3]{mio2}, there exists $\delta>0,(t_k)_{k\in\N}\subset [t_0,+\infty[$ such that $\lim_{k\rightarrow +\infty} t_k=+\infty$, 
\[ 
\norm{\nabla f(X(\omega, t_k)+\beta(t_k)V(\omega,t_k))}>\delta\quad \forall\omega\in\Omega_y \qandq t_{k+1}-t_k>1, \quad \forall k\in\N.
\]

Recall that by \eqref{H2}, there exists $\beta_0$ such that $\beta(t)\leq \beta_0$. Let $\varepsilon\in \left(0,\min\{1,\frac{\delta^2}{4L^2}\}\right)$ arbitrary and consider $Y(t)=X(t)+\beta(t)V(t)$, let also $k\in\N$ arbitrary and $t\in[t_k,t_k+\varepsilon]$. Then, using Lemma \ref{axby} and Jensen's inequality we can bound as follows:
\begin{align*}
    \Vert Y(t)-Y(t_k)\Vert^2&\leq 2\Vert X(t)-X(t_k)\Vert^2+2\Vert\beta(t)V(t)-\beta(t_k)V(t_k)\Vert^2 \\
    &\leq 2\Big\Vert \int_{t_k}^t V(s)ds\Big\Vert^2+ 2\left(\beta_0\Vert V(t)-V(t_k)\Vert+|\beta(t)-\beta(t_k)|\max\{\Vert V(t)\Vert,\Vert V(t_k)\Vert\}\right)^2\\
    &\leq 2(t-t_k)\int_{t_k}^{+\infty}\Vert V(s)\Vert^2ds\\
    &\phantom{=}+2\left(\beta_0\Vert V(t)-V(t_k)\Vert+|\beta(t)-\beta(t_k)|\max\{\Vert V(t)\Vert,\Vert V(t_k)\Vert\}\right)^2\\
    &\leq 2(t-t_k)\int_{t_k}^{+\infty}\Vert V(s)\Vert^2ds\\
    &\phantom{=}+4\beta_0^2\Vert V(t)-V(t_k)\Vert^2+4|\beta(t)-\beta(t_k)|^2\max\{\Vert V(t)\Vert,\Vert V(t_k)\Vert\}^2.
\end{align*}
By the previous point, we have that there exists $\Omega_v\in\calF$ such that $\Pro(\Omega_v)=1$ such that $$\int_{t_0}^{+\infty}\Vert V(\omega,s)\Vert^2ds<+\infty\quad \forall\omega\in\Omega_v,$$ and $k'\in \N$ such that for every $k>k'$, for all $t\in [t_k,t_k+\varepsilon]$: $$\int_{t_k}^{+\infty}\Vert V(\omega,s)\Vert^2ds<\frac{1}{6},\quad \max\{\Vert V(\omega,t)\Vert,\Vert V(\omega,t_k)\Vert\}<1, \quad \Vert V(\omega,t)-V(\omega,t_k)\Vert^2\leq \frac{\varepsilon}{12\beta_0^2}.$$
We consider an arbitrary $\omega_0\in \Omega_y\cap\Omega_v$ 
 $(\Pro(\Omega_y\cap\Omega_v)=1)$, and we let us use the abuse of notation \linebreak
$X(t)=X(\omega_0,t),V(t)=V(\omega_0,t)$, and $Y(t)=Y(\omega_0,t)$ for the rest of this proof.\\
On the other hand, $\beta$ is continuous, so there exists $\tilde{\delta}>0$ such that, if $|t-t_k|<\tilde{\delta}$, then $|\beta(t)-\beta(t_k)|<\frac{\sqrt{\varepsilon}}{2\sqrt{3}}$.\\
Therefore, letting $\varepsilon'=\min\{\varepsilon,\tilde{\delta}\}$, we have that $$\Vert \nabla f(Y(t))-\nabla f(Y(t_k))\Vert^2\leq L^2\Vert Y(t)-Y(t_k)\Vert^2\leq L^2\varepsilon\leq\frac{\delta^2}{4}, \forall k>k',\forall t\in [t_k,t_k+\varepsilon'].$$
Then, we obtain $$\Vert \nabla f(Y(t))\Vert\geq \Vert \nabla f(Y(t_k))\Vert-\Vert \nabla f(Y(t))-\nabla f(Y(t_k))\Vert\geq\frac{\delta}{2}, \forall k>k',\forall t\in [t_k,t_k+\varepsilon'].$$
This implies that $$\int_{t_0}^{+\infty}\Vert \nabla f(Y(s))\Vert^2 ds\geq \sum_{k>k'}\int_{t_k}^{t_k+\varepsilon'}\Vert\nabla f(Y(s))\Vert^2 ds\geq\sum_{k>k'}\int_{t_k}^{t_k+\varepsilon'}\frac{\delta^2}{4}=\sum_{k>k'}\frac{\varepsilon'\delta^2}{4}=+\infty.$$
Which is a contradiction, then we conclude that $$\liminf_{t\rightarrow +\infty}\Vert \nabla f(X(t)+\beta(t)V(t))\Vert=\limsup_{t\rightarrow +\infty}\Vert \nabla f(X(t)+\beta(t)V(t))\Vert=0, \quad a.s..$$
To prove the last part of \ref{v0}, we consider that $\beta(t)\leq\beta_0$, then \begin{align*}
\Vert\nabla f(X(t))\Vert&\leq\Vert\nabla f(X(t)+\beta(t)V(t))\Vert+\Vert\nabla f(X(t))-\nabla f(X(t)+\beta(t)V(t))\Vert\\
&\leq\Vert\nabla f(X(t)+\beta(t)V(t))\Vert+L\beta_0\Vert V(t)\Vert.
\end{align*}

With this bound, we can conclude that $\lim_{t\rightarrow +\infty}\Vert\nabla f(X(t))\Vert=0$ a.s.. 
\end{proof}

The following proposition states abstract bounds in expectation of \eqref{ISIHD-S}.
\begin{proposition}\label{abcdexpcomplete}
Consider the setting of Proposition \ref{abcdcomplete}, then we have that:
\begin{enumerate}[label=(\roman*)]
    \item \label{expec} 
    $\displaystyle\mathbb{E}(f(X(t)+\beta(t)V(t))-\min f)= \mathcal{O}\left(\frac{1}{a(t)}\right).$
\end{enumerate}
Moreover, if there exists $D>0, \tilde{t}>t_0$ such that $d(t)\geq D$ for $t>\tilde{t}$, then :
\begin{enumerate}[label=(\roman*),resume]
    \item \label{expbound}
    $\sup_{t>t_0}\mathbb{E}(\Vert X(t)-x^{\star}\Vert^2)<+\infty$.
    \item $\displaystyle\mathbb{E}(\Vert V(t)\Vert^2)= \mathcal{O}\left(\frac{1+b^2(t)}{c^2(t)}\right)$.
    \item  \label{last} $\displaystyle \mathbb{E}\left(f(X(t))-\min f\right)=\mathcal{O}\left(\max\Bigg\{\frac{1}{a(t)},\frac{\beta(t)\sqrt{1+b^2(t)}}{\sqrt{a(t)}c(t)},\frac{\beta^2(t)\left(1+b^2(t)\right)}{c^2(t)}\Bigg\}\right).$

\end{enumerate}
\end{proposition}
\begin{proof}
    To prove this proposition we are going to take expectation in \eqref{eqhessianpert}. 
First, we are going to bound the negative terms by $0$, denoting $E_0\eqdef \mathcal{E}(t_0)+\max\{1,L\}\int_{t_0}^{+\infty}(a(s)\beta^2(s)+c^2(s))\sigma_{\infty}^2(s)ds$, we obtain that $$\mathbb{E}(\mathcal{E}(t,X(t),V(t)))\leq E_0.$$ This implies that \begin{itemize}
    \item $\mathbb{E}(f(X(t)+\beta(t)V(t))-\min f)\leq\frac{E_0}{a(t)}$.
    \item $\mathbb{E}(\Vert b(t)(X(t)-x^{\star})+c(t)V(t)\Vert^2)\leq 2E_0$.
    \end{itemize}
    If there exists $D>0, \tilde{t}>t_0$ such that $d(t)\geq D$ for $t>\tilde{t}$, then for $t>\tilde{t}$:
    \begin{itemize}
    \item $\mathbb{E}(\Vert X(t)-x^{\star}\Vert^2)\leq\frac{2E_0}{D}$.
    \item And also, \begin{align*}
    \mathbb{E}(\Vert V(t)\Vert^2)&\leq \frac{2}{c^2(t)}[\mathbb{E}(\Vert b(t)(X(t)-x^{\star})+c(t)V(t)\Vert^2)+b^2(t)\mathbb{E}(\Vert X(t)-x^{\star}\Vert^2)]\\
    &\leq \frac{2}{c^2(t)}\pa{2E_0+\frac{2E_0b^2(t)}{D}}=\frac{4E_0}{c^2(t)}\pa{1+\frac{b^2(t)}{D}}.
    \end{align*}
    \item We bound the following term using the Descent Lemma \begin{align*}
        \mathbb{E}(f(X(t))-f(X(t)+\beta(t)V(t)))&\leq \beta(t)\sqrt{\mathbb{E}(\Vert\nabla f(X(t)+\beta(t)V(t))\Vert^2)}\sqrt{\mathbb{E}(\Vert V(t)\Vert^2)}\\
        &\phantom{=}+\frac{L}{2}\beta^2(t)\mathbb{E}(\Vert V(t)\Vert^2).
    \end{align*}
    Using Corollary \ref{2L}, we have 
    \begin{align*}
        \mathbb{E}(f(X(t))-f(X(t)+\beta(t)V(t)))&\leq \beta(t)\sqrt{2L\mathbb{E}( f(X(t)+\beta(t)V(t))-\min f)}\sqrt{\mathbb{E}(\Vert V(t)\Vert^2)}\\
        &+\frac{L}{2}\beta^2(t)\mathbb{E}(\Vert V(t)\Vert^2)\\
        &\leq 2E_0\sqrt{2L}\frac{\beta(t)}{\sqrt{a(t)}}\frac{\sqrt{1+\frac{b^2(t)}{D}}}{c(t)}+2LE_0\frac{\beta^2(t)\pa{1+\frac{b^2(t)}{D}}}{c^2(t)}\\
        &=\mathcal{O}\pa{\max\Bigg\{\frac{\beta(t)}{\sqrt{a(t)}}\frac{\sqrt{1+b^2(t)}}{c(t)},\frac{\beta^2(t)\pa{1+b^2(t)}}{c^2(t)}\Bigg\}}
    \end{align*}
    Then, we notice that \begin{align*}
       \mathbb{E}( f(X(t))-\min f)&=\mathbb{E}[f(X(t))-f(X(t)+\beta(t)V(t))]+\mathbb{E}[f(X(t)+\beta(t)V(t))-\min f]\\
        &=\mathcal{O}\pa{\max\Bigg\{\frac{\beta(t)}{\sqrt{a(t)}}\frac{\sqrt{1+b^2(t)}}{c(t)},\frac{\beta^2(t)\pa{1+b^2(t)}}{c^2(t)},\frac{1}{a(t)}\Bigg\}}.
    \end{align*}
    
\end{itemize} 
\end{proof}

\bibliographystyle{unsrt}
\smaller
\bibliography{citas}

\end{document}